\documentclass[a4paper,12pt]{article}
\usepackage{amsmath,amsthm,amsfonts,amssymb,bm,mathrsfs}
\usepackage{enumerate}
\usepackage{graphicx}
\usepackage[numbers]{natbib}
\usepackage{xcolor}
\usepackage[scale=0.75]{geometry}
\usepackage{placeins} % Control float placement

\usepackage{setspace}
\allowdisplaybreaks

\usepackage{hyperref}
\hypersetup{colorlinks=true,
           citecolor=blue,
           linkcolor=blue,
           urlcolor=blue,
           bookmarksopen=true,
           pdfstartview={XYZ null null 1.00},
           pdfpagelayout={SinglePage}
}

\newtheorem{defn}{Definition}
\newtheorem{thm}{Theorem}
\newtheorem{lem}{Lemma}

\newtheorem{rmk}{Remark}
\newtheorem{exam}{Example}
\newtheorem{coro}{Corollary}

\newcommand{\bN}{\mathbb{N}}
\newcommand{\bR}{\mathbb{R}}

\newcommand{\cB}{\mathcal{B}}

\newcommand{\cF}{\mathcal{F}}
\newcommand{\cG}{\mathcal{G}}
\newcommand{\cH}{\mathcal{H}}
\newcommand{\cI}{\mathcal{I}}

\newcommand{\cM}{\mathcal{M}}

\newcommand{\cR}{\mathscr{R}}

\newcommand{\cT}{\mathcal{T}}
\newcommand{\cU}{\mathcal{U}}

\DeclareMathOperator{\dist}{d}
\DeclareMathOperator{\rmd}{d\!}

\begin{document}
\title{The Necessity of Nowhere Equivalence\footnote{
The authors respectfully dedicate this work to Professor Peter A. Loeb on the occasion of his eightieth birthday.
They are also grateful to Xiang Sun for his help. This research was partially supported by the NUS grants R-122-000-227-112 and R-146-000-215-112.}}
\author{Wei~He\thanks{Department of Economics, The Chinese University of Hong Kong, Shatin N.T., Hong Kong. E-mail: hewei@cuhk.edu.com.}
\and
Yeneng~Sun\thanks{Department of Mathematics, National University of Singapore, 10 Lower Kent Ridge Road, Singapore 119076. E-mail:  ynsun@nus.edu.sg.}
}
\date{This version: \today}
\maketitle

\abstract{We prove some regularity properties (convexity, closedness, compactness and preservation of upper hemicontinuity) for distribution and regular conditional distribution of correspondences under the nowhere equivalence condition. We show the necessity of such a condition for any of these properties to hold. As an application, we demonstrate that the nowhere equivalence condition is satisfied on the underlying agent space if and only if pure-strategy Nash equilibria exist in general large games with any fixed uncountable compact action space.
\bigskip


}

\newpage
\tableofcontents
\newpage

\section{Introduction}\label{sec-intro}

The theory of correspondences, which has important applications in a variety of areas (including optimization, control theory and mathematical economics), has been studied extensively. However, basic regularity properties on the distribution of correspondences such as convexity, closedness, compactness and preservation of upper hemicontinuity may all fail when the underlying probability space is the Lebesgue unit interval; see, for example, \cite{KS2009} and \cite{Sun1996}. These issues were resolved in \cite{Sun1996} by considering a class of rich measure spaces, the so-called Loeb measure spaces constructed from the method of nonstandard analysis.\footnote{See \cite{Loeb1975} and \cite{LW2015} for the construction of Loeb spaces.} It was further shown in \cite{KS2009} that the abstract property of saturation\footnote{As noted in \cite{HK1984}, atomless Loeb probability spaces are saturated.
For some other applications of Loeb and saturated probability spaces, see, for example, \cite{CP2014}, \cite{Jin2015}, \cite{DS2007}, \cite{FK1996}, \cite{Keisler1997}, \cite{KSagara2016}, \cite{LS2007}, \cite{Sagara2015}, \cite{Sun2016}, and \cite{SZ2015}.} on a probability space is not only sufficient but also necessary for any of these regularity properties to hold.\footnote{When the target space is a Banach space, one can also consider Bochner and Gelfand integration of correspondences. The same kind of regularity properties were shown to hold under Loeb/saturated probability spaces in \cite{Podczeck2008}, \cite{Sun1997} and \cite{SY2008}, while the necessity of saturation for these properties was indicated in \cite{Podczeck2008} and \cite{SY2008}. A related issue on the purification of measure-valued maps on Loeb/saturated probability spaces was considered in \cite{LS2006}, \cite{LS2009} and \cite{Podczeck2009} with the necessity of saturation in \cite{LS2009} and \cite{Podczeck2009}.}

Theorem 3B.7 of \cite[p. 47]{FK2002} by Fajardo and Keisler indicated that a saturated probability space is necessarily rich with measurable sets (also \cite[Corollary 4.5]{HK1984} by Hoover and Keisler implicitly) in the sense that any of its nontrivial sub-measure space is not countably generated module null sets. However, standard probability spaces such as complete separable metric spaces with Borel probability measures are only countably generated (thus not saturated). To allow the possibility of working with such standard  probability spaces, this paper uses the condition of ``nowhere equivalence'' to characterize some general results on correspondences and the existence of Nash equilibrium in large games.

As noted in  \cite{HSS2017}, the nowhere equivalence condition was motivated by the fact that various equilibrium properties in economics may require different agents with the same characteristic to choose different actions. Thus, one needs to distinguish the $\sigma$-algebra $\cT$ in a probability space $(T,\cT,\lambda)$ (modeling the space of agents) from the $\sigma$-algebra $\cF$ generated by the mapping specifying the individual characteristics. The condition captures the idea that for any nontrivial collection $D$ of agents, the $\sigma$-algebra $\cT$ is \textit{richer} than its sub-$\sigma$-algebra $\cF$ when they are restricted to $D$. Formally, let $(T, \cT, \lambda)$ be an atomless probability space and $\cF$ a countably generated sub-$\sigma$-algebra of $\cT$. Then $\cT$ is said to be nowhere equivalent to $\cF$ if for any $D \in \cT$, the restrictions of the two $\sigma$-algebras to $D$ do not coincide.\footnote{See Definition 1 in \cite{HSS2017} and Definition \ref{defn-nowhere_equiv} in Section \ref{sec-defn} below.}

The first aim of this paper is to study the regularity properties on the distribution of a correspondence. Let $(T, \cT, \lambda)$ be an atomless probability space and $\cF$ a countably generated sub-$\sigma$-algebra of $\cT$. If a correspondence $G$ is $\cF$-measurable and its selections are allowed to be $\cT$-measurable, then Theorem~\ref{thm-distribution} below shows that any of the regularity properties on the distribution of a correspondence as discussed above holds if and only if $\cT$ is nowhere equivalent to $\cF$. It is easy to see that an atomless probability space is saturated if and only if it is nowhere equivalent to any countably generated sub-$\sigma$-algebra. Thus, Theorems~\ref{thm-distribution} goes beyond the characterization results in \cite{KS2009} for saturated probability spaces to cover the case of non-saturated probability spaces such as the standard probability spaces. We may point out that to obtain the necessity of nowhere equivalence, one needs to construct considerably more complicated counterexamples than those used in \cite{KS2009}, while the necessity of saturation can often be obtained by modifying a counterexample based on the Lebesgue measure space to a non-saturated probability space via a measure-preserving mapping.\footnote{See \cite{KS2009}, \cite{KRSY2013} and \cite{QY2014}.}

The second aim of this paper is to extend those regularity properties to regular conditional distribution of correspondences.\footnote{To the best of our knowledge, this paper is the first to consider the regular conditional distribution of correspondences.} Based on the nowhere equivalence condition, we show that the set of regular conditional distributions induced by all the $\cT$-measurable selections of an $\cF$-measurable correspondence will satisfy the usual regularity properties of convexity, closedness, compactness and preservation of upper hemicontinuity. Since the results for distribution of correspondences are special cases of those results for the regular conditional distribution of correspondences, the necessity of nowhere equivalence thus follows from any of the regularity properties on regular conditional distribution of a correspondence.

The third aim of this paper is to consider an application in large games. The necessity of nowhere equivalence has been shown for the existence of pure strategy equilibria in a sequence of large games with the same agent space but different action spaces in \cite{HSS2017}. If one restricts attention to large games with a particular uncountable action space such as the unit interval, it is not known whether the nowhere equivalence condition is necessary. However, Theorem 4.6 of \cite{KS2009} does show that for any fixed uncountable compact action space, the saturation property is necessary and sufficient for the existence of pure strategy equilibria in general large games with such an action space. Theorem \ref{thm-lg} below extends this result to the case of nowhere equivalence.

The rest of the paper is organized as follows. Some basic definitions are given in Section~\ref{sec-defn}. Results on distribution and regular conditional distribution of correspondences are presented in Sections~\ref{sec-defi} and \ref{sec-RCD}, respectively. An application to large games is provided in Section~\ref{sec-lg}. Section~\ref{sec-proof} contains the proofs of Theorems~\ref{thm-distribution} and \ref{thm-rcd}. The proof of Theorem~\ref{thm-lg} is left in Section~\ref{sec-game proof}.

\section{Basics}\label{sec-defn}

Throughout this paper, the triple $(T,\cT,\lambda)$ denotes an atomless probability space with a complete countably additive probability measure $\lambda$. A correspondence $F$ from $T$  to a topological space $X$  is a mapping from $T$ to the set of nonempty subsets of $X$; that is, $\mathcal{P}(X) \setminus \emptyset$. A correspondence $F$ is said to be measurable if its graph belongs to the product $\sigma$-algebra $\cT\otimes\mathcal{B}(X)$, where $\cB(X)$ is the Borel $\sigma$-algebra of $X$. A mapping $f:T\to X$ is called a selection of $F$ if $f(t)\in F(t)$ for $\lambda$-almost all $t\in T$. The correspondence $F$ is said to be closed (resp. compact, convex) valued if $F(t)$ is a closed (resp. compact, convex) subset of $X$ for $\lambda$-almost all $t\in T$.

A correspondence $G$ from a Polish space $Y$ (\textit{i.e.}, a complete separable metrizable topological space) to another Polish space $Z$ is said to be upper hemicontinuous at $y_0\in Y$ if for any open set $O_Z$ that contains $G(y_0)$, there exists an open neighborhood $O_Y$ of $y_0$ such that for any $y \in O_Y$, $G(y)\subseteq O_Z$. The correspondence $G$ is said to be upper hemicontinuous if it is upper hemicontinuous at every point $y\in Y$.

Hereafter, the usual Lebesgue unit interval is denoted by $(I,\cI,\eta)$; that is, the unit interval $I=[0,1]$ is endowed with the Lebesgue  $\sigma$-algebra $\cI$ (i.e., the completion of the Borel $\sigma$-algebra) and the Lebesgue measure $\eta$. Let $X$ and $Y$ be  Polish spaces, and $\cM(X)$ (resp. $\cM(Y)$) the space of all Borel probability measures endowed with the topology of weak convergence on $X$ (resp. $Y$). Then $\cM(X)$ and $\cM(Y)$ are again Polish spaces.

Let $\cF$ be a countably generated\footnote{A probability space (or its $\sigma$-algebra) is said to be countably generated if its $\sigma$-algebra is generated by countable measurable subsets together with the null sets.} sub-$\sigma$-algebra of $\cT$. For any nonnegligible set $E\in\cT$, the restricted probability space $(E,\cF^E,\lambda^E)$ on $E$ is defined as follows: $\cF^E$ is the $\sigma$-algebra $\{ E\cap E' \colon E' \in \cF \}$ and $\lambda^E$ is the probability measure re-scaled from the restriction of $\lambda$ to $\cF^E$. The restricted space $(E,\cT^E,\lambda^E)$ can be defined similarly.

The following condition of ``nowhere equivalence'' is in Definition 1 of \cite{HSS2017}.

\begin{defn}\label{defn-nowhere_equiv}
The $\sigma$-algebra $\cT$ is said to be \textbf{nowhere equivalent} to $\cF$ if for every $D\in\cT$ with $\lambda(D)>0$, there exists a $\cT$-measurable subset $D_0$ of $D$ such that $\lambda(D_0\triangle D_1)>0$ for any $D_1\in\cF^D$.
\end{defn}

\section{Distribution of correspondences}\label{sec-defi}

It was shown in \cite{Sun1996} that the regularity properties (convexity, closedness, compactness, and preservation of upper hemicontinuity) for distribution of correspondences hold if the underlying probability space is an atomless Loeb space. In \cite{KS2009}, it is shown that these properties hold if and only if the underlying probability space is saturated. The saturation condition implies that the underlying probability space is sufficiently rich (nowhere countably generated, see Fact~2.5 therein). Since an atomless Loeb space is always saturated, the result in \cite{KS2009} extends that in \cite{Sun1996}.

In various applications, the widely adopted probability spaces are based on Polish spaces (\textit{e.g.}, the Lebesgue unit interval), which are  not saturated. As a result, those regularity properties may fail. In this section, we shall distinguish the measurability of a correspondence and the measurability of its selections by working with the nowhere equivalence condition, and are able to characterize the regularity properties for the distribution of correspondences. Since the $\sigma$-algebra in a saturated probability space is nowhere equivalent to any countably generated sub-$\sigma$-algebra, we thus extend the corresponding characterization results in Section~3 of \cite{KS2009}.

We follow the notation as specified in Section \ref{sec-defn}. Let $(T,\cT,\lambda)$ be an atomless probability space and $\cF$ a countably generated sub-$\sigma$-algebra of $\cT$, and $X$ a Polish space. For a correspondence $F$ from $T$ to  $X$, let
$$D_{F}^{\cT}=\{\lambda f^{-1}\in\cM(X): f\text{ is a }\cT\text{-measurable  selection of }F\},$$
where $\lambda f^{-1}$ is the induced distribution of $f$ on $X$.

Our first theorem characterizes the regularity properties for the distribution of correspondences via the nowhere equivalence condition.

\begin{thm}\label{thm-distribution}
Let $X$ be a fixed uncountable Polish space.\footnote{The necessity result in Theorem 3.7 of \cite{KS2009} is proved by working with the special case that $X$ is an interval on the real line. Here our necessity result is stated in terms of any fixed uncountable Polish space.} The condition that $\cT$ is nowhere equivalent to $\cF$ is necessary and sufficient  for each of the following properties.
\begin{description}
  \item[A1] For any closed valued $\cF$-measurable correspondence $F$ from $T$ to $X$, $D_{F}^{\cT}$ is convex.

  \item[A2] For any closed valued  $\cF$-measurable correspondence $F$ from $T$ to $X$, $D_{F}^{\cT}$ is closed.

  \item[A3] For any compact valued $\cF$-measurable correspondence $F$ from $T$ to $X$, $D_{F}^{\cT}$ is compact.

  \item[A4] For any closed valued correspondence $G$ from $T\times Y$ to $X$ ($Y$ is a Polish space) such that there exists a compact valued $\cF$-measurable correspondence $F$ from $T$ to $X$ and
      \begin{description}
        \item[a] for any $(t,y)\in T\times Y$, $G(t,y)\subseteq F(t)$,
        \item[b] for any $y\in Y$, $G(\text{\ensuremath{\cdot}},y)$ (denoted as $G_y$) is $\cF$-measurable from $T$ to $X$,
        \item[c] for any $t\in T$, $G(t,\cdot)$ (denoted as $G_t$) is upper hemicontinuous from $Y$ to $X$,
      \end{description}
      the correspondence $H(y)=D_{G_{y}}^{\cT}$ is upper hemicontinuous from $Y$ to $\cM(X)$.

  \item[A5] For any $\cF$-measurable mapping $G$ from $T$ to $\mathcal{M}(X)$, there exists a $\cT$-measurable mapping $f$ from $T$ to $X$ such that
      \begin{enumerate}
        \item for every Borel subset $B$ of $X$,
            $$\lambda f^{-1}(B)=\int_{T}G(t)(B)\rmd\lambda(t);$$
        \item $f(t)\in supp\, G(t)$ for $\lambda$-almost all $t\in T$, where \textup{$supp\, G(t)$} is the support of the probability measure $G(t)$ on $X$ (the support of a Borel probability measure on $X$ is the smallest closed set in $X$ with measure one).
      \end{enumerate}

\end{description}
\end{thm}

\section{Regular conditional distribution of correspondences}\label{sec-RCD}

Since we distinguish the measurability of a correspondence and the measurability of its selections, it is natural to consider the set of regular conditional distributions, each of which is induced by a selection of the correspondence conditioned on the smaller $\sigma$-algebra. In this section, we characterize the regularity properties of the regular conditional distribution of correspondences via the nowhere equivalence condition.

As in the previous section, $(T,\cT,\lambda)$ is an atomless probability space and $\cF$ a countably generated sub-$\sigma$-algebra of $\cT$. For a Polish space $X$, let $C_b(X)$ be the space of all bounded continuous functions from $X$ to the space $\bR$ of real numbers. A real-valued function $c$ on $T \times X$ is said to be a Carath\'eodory function if $c(\cdot,x)$ is $\cF$-measurable for each $x\in X$, and $c(t,\cdot)$ is continuous for each $t\in T$.

The following is the standard definition of transition probabilities; see, for example, \cite{Bogachev2007}.

\begin{defn}\label{defn-young measure}
An $\cF$-measurable transition probability from $T$ to $X$ is a mapping $\phi:T\to \cM(X)$ such that $\phi(\cdot,B): t\to \phi(t,B)$ is $\cF$-measurable for every $B\in\cB(X)$.

The set of all $\cF$-measurable transition probabilities from $T$ to $X$ is denoted by $\cR^\cF(X)$ (or $\cR^\cF$ when there is no confusion).
\end{defn}

\begin{defn}\label{defn-convergence rcd}
The weak topology on $\cR^\cF$ is defined as the weakest topology for which the functional $\phi\to  \int_T[\int_X c(t,x) \phi(t,\rmd x)] \rmd\lambda(t)$  is continuous for every  bounded Carath\'eodory function $c$ on $T \times X$.
\end{defn}

Let $f$ be a $\cT$-measurable mapping from $T$ to $X$, and $\mu^{f|\cF}$ the regular conditional distribution (RCD) of $f$ given $\cF$.\footnote{Since $X$ is a Polish space endowed with the Borel $\sigma$-algebra, the RCD $\mu^{f|\cF}$ always exists; see \cite{Durrett2010}.} That is, $\mu^{f|\cF}$ is a mapping from $T\times \cB(X)$ to $[0,1]$ such that
\begin{enumerate}
  \item $\mu^{f|\cF}(t,\cdot)$ is a Borel probability measure on $X$ for each $t\in T$;
  \item given any Borel subset $B$ of $X$, $\mu^{f|\cF}(\cdot, B)$ is the conditional expectation $\mathbb{E} \left[1_B(f)|\cF \right]$ of the indicator function $1_{f^{-1}(B)} =1_B(f)$, given $\cF$.
\end{enumerate}
Let $F$ be a correspondence from $T$ to $X$. We denote
$$\cR^{(\cT,\cF)}_F=\{\mu^{f|\cF}: f \mbox{ is a } \cT\mbox{-measurable selection of } F\},$$
which is the set of all RCDs induced by the $\cT$-measurable selections of $F$ conditioned on $\cF$.

Below, we show that the corresponding results of Theorem \ref{thm-distribution} still hold in the setting of regular conditional distributions. In particular, the sufficiency part of Theorem~\ref{thm-rcd} generalizes the sufficiency part of Theorem~\ref{thm-distribution}, where the distribution of a selection can be viewed as the regular conditional distribution of the selection with respect to the trivial $\sigma$-algebra $\{T, \emptyset\}$.

\begin{thm}\label{thm-rcd} Let $X$ be a fixed uncountable Polish space.
The condition that $\cT$ is nowhere equivalent to $\cF$ is necessary and sufficient  for the validity of each of the following properties under any sub-$\sigma$-algebra $\cG$ of $\cF$.
\begin{description}
  \item[B1] For any closed valued $\cF$-measurable correspondence $F$ from $T$ to $X$, $\cR^{(\cT,\cG)}_F$ is convex.

  \item[B2] For any closed valued $\cF$-measurable correspondence $F$ from $T$ to $X$, $\cR^{(\cT,\cG)}_F$ is weakly closed.

  \item[B3] For any compact valued $\cF$-measurable correspondence $F$ from $T$ to $X$, $\cR^{(\cT,\cG)}_F$ is weakly compact.

  \item[B4] For any closed valued correspondence $G$ from $T\times Z$ to $X$ such that there exists a compact valued correspondence $F$ from $T$ to $X$ and
      \begin{description}
        \item[a] for any $(t,z)\in T\times Z$, $G(t,z)\subseteq F(t)$;
        \item[b] for any $z\in Z$, $G(\text{\ensuremath{\cdot}},z)$ is $\cF$-measurable from $T$ to $X$;
        \item[c] for any $t\in T$, $G(t,\cdot)$ is upper hemicontinuous from $Z$ to $X$;
      \end{description}
      The correspondence $H(z)=\cR_{G_{z}}^{(\cT,\cG)}$ is upper hemicontinuous from $Z$ to $\cR^{\cG}$.

  \item[B5] For any $G\in \cR^\cG$, there exists a $\cT$-measurable mapping $g$ such that $\mu^{g|\cG}=G$.\footnote{The sufficiency part of the property (B5) was proved in Lemma~4.4~(iii) of \cite{HK1984}.  We shall give an alternative proof in the appendix.}
  \end{description}
\end{thm}

\section{Large games}\label{sec-lg}

In this section, we shall study the existence of pure strategy equilibria in large games by adopting the nowhere equivalence condition.

The agent space is modeled by an atomless probability space $(T,\cT,\lambda)$. Let $K$ be a compact metric space that serves as the common action space for all the players, and $\cM(K)$ the space of all Borel probability measures on $K$ endowed with the topology of weak convergence of measures. Denote $\cU$ as the space of bounded continuous real-valued functions on $K\times\cM(K)$, which is the space of all possible payoff functions.

\begin{defn}\label{defn-lg}
A large game is a measurable mapping $G$ from $(T,\cT,\lambda)$ to $\cU$.
A Nash equilibrium in the large game $G$ is a measurable mapping $g$ from $(T,\cT,\lambda)$ to $K$ such that for $\lambda$-almost all $t\in T$,
$$G(t)(g(t),\lambda g^{-1})\ge G(t)(k,\lambda g^{-1})\text{ for all $k \in K$.} $$
\end{defn}

Let $\cF$ be a sub-$\sigma$-algebra of $\cT$, which can be understood as the $\sigma$-algebra generated by the function $G$ that specifies the agents' characteristics. Such a probability space $(T,\cF,\lambda)$ is called the characteristic type space in \cite{HSS2017}.

It was demonstrated in \cite{RSY1995} that a large game may not possess any Nash equilibrium via a rather involved example in which the agent space is the Lebesgue unit interval.\footnote{A countably generated Lebesgue extension, which is nowhere equivalent to the relevant Borel $\sigma$-algebra, is presented in \cite{KZ2012} as the agent space such that the example in \cite{RSY1995} has a Nash equilibrium. That result of \cite{KZ2012} also follows from Theorem~2 of \cite{HSS2017}.} This nonexistence problem was resolved in \cite{KS1999} by working with a hyperfinite Loeb agent space. It was then shown in Theorem 4.6 of \cite{KS2009} that the saturation property is necessary and sufficient  for the existence of pure strategy Nash equilibria in large games with any fixed uncountable compact metric space as the action space. A similar characterization on the existence of pure strategy Nash equilibria was considered in \cite{HSS2017}. While the strategy profiles and agents' characteristics in \cite{KS2009} are measurable with respect to the $\sigma$-algebra from an underlying agent space, say $(T,\cT,\lambda)$, it is assumed in \cite{HSS2017} that  the strategy profiles are measurable with respect to $\cT$, but the agents' characteristics are measurable with respect to a sub-$\sigma$-algebra $\cF$.  It was shown in \cite[Theorem~2]{HSS2017}  that any $\cF$-measurable game has a $\cT$-measurable pure strategy Nash equilibrium if and only if $\cT$ is nowhere equivalent to $\cF$. A sequence of large games that have the same agent space but different action spaces is constructed in \cite{HSS2017} for proving the necessity part. However, the necessity of nowhere equivalence for the existence result on Nash equilibrium is unknown if one restricts the attention to large games with any specific uncountable action space, say, the unit interval.

In the following theorem, we show that the nowhere equivalence condition is indeed necessary and sufficient for the existence of pure strategy Nash equilibria in large games with any fixed uncountable compact metric space as the action space. The sufficiency part of this theorem is a special case of the corresponding result in Theorem 2 of \cite{HSS2017}. The necessity part will be proved in Section \ref{sec-game proof}.

\begin{thm}\label{thm-lg}
Let $K$ be a fixed uncountable compact metric space. Any $\cF$-measurable game with the action space $K$ has a $\cT$-measurable Nash equilibrium if and only if $\cT$ is nowhere equivalent to $\cF$.\footnote{A more general class of large games was considered in \cite{KRSY2013}, where the individual agents have names as well as traits. It was shown in \cite{KRSY2013} that the saturation property is necessary and sufficient for the existence of pure strategy Nash equilibria and the closed graph property; see also \cite{QY2014}. That result for large games with traits was extended in \cite{HSS2017} by adopting the nowhere equivalence condition, where the necessity part was shown via a sequence of large games with the same agent space but different action spaces. It is clear that the necessity part of our Theorem~\ref{thm-lg} can also be regarded as a necessity result for large games with traits and but with any {\it fixed} uncountable compact metric space as the action space (the trait space is a singleton set).}
\end{thm}

\section{Proofs of Theorems~\ref{thm-distribution} and \ref{thm-rcd}}\label{sec-proof}

\subsection{Preliminary lemmas}

In the following, we state several conditions which are equivalent to the condition of nowhere equivalence.

\begin{defn}\label{defn-others}
\begin{enumerate}
\item The $\sigma$-algebra $\cT$ is \textbf{conditional atomless} over $\cF$ if for every $D\in\cT$ with $\lambda(D)>0$, there exists a $\cT$-measurable subset $D_0$ of $D$ such that on some set of positive probability,
$$0<\lambda(D_0\mid\cF)<\lambda(D\mid\cF).$$

\item The $\sigma$-algebra $\cT$ is said to be \textbf{relatively saturated} with respect to $\cF$ if for any Polish spaces $X$ and $Y$, any measure $\mu\in\cM(X\times Y)$, and any $\cF$-measurable mapping $f$ from $T$ to $X$ with $\lambda f^{-1}=\mu_{X}$,\footnote{Hereafter, $\mu_{X}$ denotes the marginal of $\mu$ on the space $X$.} there exists a $\cT$-measurable mapping $g$ from $T$ to $Y$ such that $\mu=\lambda(f,g)^{-1}$.

\item The $\sigma$-algebra $\cF$ admits an \textbf{atomless independent supplement} in $\cT$ if there exists another sub-$\sigma$-algebra $\cH$ of $\cT$ such that $(T,\cH,\lambda)$ is atomless, and for any $C_1\in\cF$ and $C_2\in\cH$, $\lambda(C_1\cap C_2)=\lambda(C_1)\lambda(C_2)$.

\item The $\sigma$-algebra $\cF$ admits an \textbf{asymptotic independent supplement} in $\cT$ if for some strictly increasing sequence $\{n_k\}_{k=1}^\infty$ of positive integers and for each $k \ge 1$, there exists a $\cT$-measurable partition $\{E_1,E_2,\ldots,E_{n_k}\}$ of $T$ with $\lambda(E_j) = \frac{1}{n_k}$ and $E_j$ being independent of $\cF$ for $j=1,2,\ldots, n_k$.
\end{enumerate}
\end{defn}

Condition~(1), which is called ``$\cT$ is atomless over $\cF$'' in Definition~4.3 of \cite{HK1984}, is a generalization of the notion of atomlessness. The concept of relative saturation refines the concept of saturation; see \cite{HK1984} and \cite{KS2009} for the formal definition of saturation. Condition~(3) implies that there exist sufficiently many independent events for the $\sigma$-algebra $\cF$ within the $\sigma$-algebra $\cT$. Condition~(4) is an asymptotic version of Condition~(3), which will be used extensively for deriving the necessity parts of Theorems \ref{thm-distribution} and \ref{thm-lg}.

The following lemma shows that these conditions are indeed equivalent to the condition of nowhere equivalence. The relationship that $(iv) \Rightarrow (ii)$ was proved in Lemma~4.4 of \cite{HK1984}. See Lemmas~2 and 7 in \cite{HSS2017} for the complete proof.

\begin{lem}\label{lem-neq}
The following statements are equivalent.
\begin{enumerate}[(i)]
\item $\cT$ is nowhere equivalent to $\cF$.
\item $\cT$ is conditional atomless over $\cF$.
\item $\cT$ is relatively saturated with respect to $\cF$.
\item $\cF$ admits an atomless independent supplement in $\cT$.
\item $\cF$ admits an asymptotic independent supplement in $\cT$.
\end{enumerate}
\end{lem}

It is shown in \cite[Section~4]{HSS2017} that the condition of nowhere equivalence unifies several conditions that have been imposed on the spaces of economic agents in various applications. The following definition states one more such condition, which was called ``many agents of every type'' in \cite{Noguchi2009} and \cite{Podczeck1997}.

\begin{defn}
\label{defn-disintegration}
Let $(T, \cT, \lambda)$ be a probability space, and $X$ a Polish space. A system of regular conditional probabilities $\{\lambda_x \colon x\in X\}$  is said to be generated by a measurable mapping $G$ from $T$ to $X$ if
\begin{enumerate}
\item for every $x\in X$, $\lambda_x$ is a probability measure on $(T,\cT)$;
\item for every $D\in\cT$, $\lambda(\cdot,D)$ is measurable with respect to $\cB(X)$;
\item $\lambda(G^{-1}(E) \cap D)=\int_E \lambda_x(D) \rmd \kappa$ for each $D\in\cT$ and $E\in \cB(X)$, where $\kappa=\lambda G^{-1}$.
\end{enumerate}

The family of conditional probabilities $\{\lambda_x \colon x\in X\}$ is called proper if $\lambda_x$ is concentrated on $G^{-1}(\{x\})$ for $\kappa$-almost all $x\in X$, and atomless if $\lambda_x$ is atomless for $\kappa$-almost all $x\in X$.
\end{defn}

The following lemma shows that given a measurable mapping $G$, if the family of regular conditional probability generated by $G$ is well behaved (i.e., proper and atomless), then the nowhere equivalence condition is satisfied.

\begin{lem}\label{lem-nowhere_noguchi}
Suppose that $\cT$ is countably generated and the family of regular conditional probability  $\{\lambda_x\colon x\in X\}$ generated by $G$ is proper and atomless. Then the $\sigma$-algebra $\cG$  generated by $G$ admits an atomless independent supplement in $\cT$.\footnote{This lemma is essentially the same as Theorem 10.8.3 of \cite{Bogachev2007}. A stronger conclusion is obtained there under stronger conditions.}
\end{lem}

\begin{proof}
Since $\cT$ is countably generated, based on Theorem~6.5.5 in \cite{Bogachev2007}, there is a Borel measurable mapping $\phi$ from $(T,\cT)$ to $I=[0, 1]$ such that $\phi$ could generate the $\sigma$-algebra $\cT$. Define a mapping $f\colon X\times I \to [0,1]$ by letting $f(x,i)=\lambda_x \phi^{-1}([0,i])$ for any $x \in X$ and $i \in I$. For each $x\in X$, denote $f_x(\cdot)=f(x,\cdot)$, and hence $f_x$ is the distribution function of the probability measure $\lambda_x \phi^{-1}$ on $I$.
For $\kappa$-almost all $x\in X$, the atomlessness of $\lambda_x$ on $\cT$ implies that $\lambda_x \phi^{-1}(\{i\})=0$ for all $i\in I$. Thus the distribution function $f_x(\cdot)$ is continuous on $I$ for each $x\in X$.

Let $g(t)=f(G(t), \phi(t))$ for each $t \in T$. We claim that $g$ is independent of $G$ and $\lambda  g^{-1}$ is the Lebesgue measure $\eta$ on $I$.
Firstly, by condition (3) of Definition \ref{defn-disintegration}, we have $\lambda(G^{-1}(E) \cap g^{-1}([a,b]))=\int_E \lambda_x\big(g^{-1}([a,b]) \big) \rmd \kappa$ for any $E\in \cB(X)$ and $0\le a < b \le 1$.
Secondly, the condition that $\{\lambda_x\colon x\in X\}$ is proper implies that $\lambda_x$ is concentrated on $G^{-1}(\{x\})$, and hence $\lambda_x\big(g^{-1}([a,b]) \big)=\lambda_x\big((f_x\circ \phi)^{-1}([a,b]) \big)$ for $\kappa$-almost all $x\in X$.
Thirdly, for $\kappa$-almost all $x\in X$, $f_x$ is the continuous distribution function of the probability measure $\lambda_x  \phi^{-1}$, it then follows by Example~3.6.2 of \citet[p.~192]{Bogachev2007} that $\lambda_x \phi^{-1} f_x^{-1}=\eta$.
Thus we have $\lambda(G^{-1}(E) \cap g^{-1}([a,b]))=\kappa(E)\eta([a,b])$.

Therefore, $g$ is independent of $G$ and induces an atomless $\sigma$-algebra, which yields the assertion.
\end{proof}

\begin{rmk}
The converse direction of Lemma~\ref{lem-nowhere_noguchi} may not be true. Let $(I,\cI',\eta)$ be a countably generated extension of the Lebesgue unit interval $(I,\cI,\eta)$ such that $\cI$ admits an atomless independent supplement in $\cI'$. Let $(T,\cT, \lambda)=(I, \cI',\eta)$  and $X=I$. Take $G$ as the identity mapping on $I$; \textit{i.e.}, $G(i)=i$ for all $i\in I$. Then $\cG=\cI$, and $\cG$ admits an atomless independent supplement in $\cT$. It is easy to see that $\lambda_i=\delta_i$ (the Dirac measure at point $i$) for $i\in I$ defines  the regular conditional probability generated by $G$. For almost all $i\in I$, $\lambda_i$ is purely atomic.
\end{rmk}

Every Polish space admits a (not necessarily complete) totally bounded metric (see \cite{Billingsley1968}). Let $d$ be a totally bounded metric  on the Polish space $X$. The distance between a point $a\in X$ and a nonempty set $B\subseteq X$ is defined by
$$d(a,B)=\inf_{b\in B}d(a,b).$$
The Hausdorff semidistance between two nonempty sets $A$ and $B$ is $\sigma(A,B) = \sup_{a\in A}d(a,B)$. The Hausdorff distance between two sets $A$ and $B$ is defined as $\rho(A,B)=\max\{\sigma(A,B),\sigma(B,A)\}$.
Let $\cF_X$ be the hyperspace of nonempty closed subsets of $X$ endowed with the topology induced by the metric $\rho$.

The characterization of the measurability of correspondences with compact range in terms of mappings taking values in the hyperspace $\cF_X$ is well known. The following lemma from \cite{Sun1996} considers the case with the range of Polish spaces.
\begin{lem}\label{lem-hyperspace}
\begin{enumerate}
  \item The hyperspace $\cF_X$ with the metric $\rho$ is a Polish space.
  \item If $F$ is a closed valued correspondence from $T$ to $X$, then $F$ is a measurable correspondence if and only if $F$ is a measurable mapping from $T$ to $\cF_X$.
\end{enumerate}
\end{lem}

\subsection{Proof of the sufficiency part of Theorem~\ref{thm-distribution}}\label{subsec-proof S1}

Suppose that $(T,\cF,\lambda)$ is not atomless. Then $T$ can be written as the union of disjoint $\cF$-measurable sets $T_1$ and $T_2$ so that the restricted $\sigma$-algebras $\cF^{T_1}$ and $\cF^{T_2}$ are respectively atomless and purely atomic under the measure $\lambda$. On $T_2$, we can find a countably generated atomless $\sigma$-algebra $\cF^2$ such that $\cF^2\subseteq \cT^{T_2}$ and $\cT^{T_2}$ is nowhere equivalent to $\cF^2$; see Lemma 1 in \cite{HSS2017}. Let $\cF'$ be the $\sigma$-algebra generated by the measurable sets in $\cF^{T_1}$ and $\cF^2$. Then $\cF\subseteq\cF'$, $\cF'$ is atomless and countably generated, and $\cT$ is nowhere equivalent to $\cF'$. Any $\cF$-measurable correspondence must also be measurable with respect to $\cF'$. Therefore, we only need to prove the case that $(T,\cF,\lambda)$ is an atomless probability space.

Let $X$ be any Polish space, and $\Gamma=\{(x,E)\in X\times\mathcal{F}_{X}:x\in E\}$. It is clear that $\Gamma$ is a closed set in $X\times\mathcal{F}_{X}$. Then $f$ is a $\cT$-measurable selection of a closed valued $\cT$-measurable correspondence $F$ if and only if $\lambda(f,F)^{-1}(\Gamma)=1$.

\

A1. For any $\mu_{1}, \mu_{2}\in D_{F}^{\cT}$ and $0\leq\alpha\leq1$, let $\mu=\alpha\mu_{1}+(1-\alpha)\mu_{2}$. There are two $\cT$-measurable selections $f_1$ and $f_2$ of $F$ such that $\mu_i=\lambda f_i^{-1}$ for $i=1,2$.
Denote $\tau_i=\lambda(f_i,F)^{-1}$ for $i=1,2$ and $\tau=\alpha\tau_{1}+(1-\alpha)\tau_{2}$. Then $\tau_X=\alpha\mu_{1}+(1-\alpha)\mu_{2}=\mu$ and $\tau_{\cF_X}=\lambda F^{-1}$.

By Lemma~\ref{lem-neq}, $\cT$ is relatively saturated with respect to $\cF$. Thus, there exists a $\cT$-measurable function $f$ such that $\lambda(f,F)^{-1}=\tau$. Since $\tau_i(\Gamma)=1$ for $i=1,2$, we have $\tau(\Gamma)=1$, and $f$ is a $\cT$-measurable selection of $F$. As a result,
$$\alpha\mu_{1}+(1-\alpha)\mu_{2} = \tau_X = \lambda f^{-1} \in D_{F}^{\cT},$$
which implies that $D_{F}^{\cT}$ is convex.

\

A2.To prove the closedness of  $D_{F}^{\cT}$, we claim that $D_{F}^{\cT}=\overline{D_{F}^{\cF}}$, where
$$D_{F}^{\cF}=\{\lambda f^{-1}\in\cM(X): f\text{ is an }\cF\text{-measurable  selection of }F\},$$ and $\overline{D_{F}^{\cF}}$ denotes the closure of  $D_{F}^{\cF}$.

First, we show the direction $D_{F}^{\cT}\subseteq\overline{D_{F}^{\cF}}$. For any $\mu\in D_{F}^{\cT}$, there exists a $\cT$-measurable selection $f$ of $F$ such that $\mu=\lambda f^{-1}$.  For any open set $O\subseteq X$,
$$\mu(O)=\lambda(t:f(t)\in O)\leq \lambda(t:F(t)\cap O\neq\emptyset)=\lambda(F^{-1}(O)).$$
By Proposition~3.5 in \cite{KS2009}, the Borel probability measure $\mu$ belongs to the closure of $D_{F}^{\cF}$ if and only if
\begin{equation}
\mu(O)\leq \lambda(F^{-1}(O))\text{ for any open set O in X}.\label{eq:Borel condition}
\end{equation}
As a result, $\mu\in\overline{D_{F}^{\cF}}$, which implies that $D_{F}^{\cT}\subseteq\overline{D_{F}^{\cF}}$.

Conversely, for any $\mu\in\overline{D_{F}^{\cF}}$, there exists a sequence $\{\mu_{n}\}_{n=1}^{\infty}\subseteq D_{F}^{\cF}$ such that $\mu_{n}\rightarrow\mu$ weakly. For each $n \ge 1$, let $f_n$ be an  $\cF$-measurable selection of $F$ such that $\mu_{n}=\lambda f_{n}^{-1}$.
Since the family $\{\mu_1,\cdots,\mu_n,\cdots\}$ is tight, the sequence $\{\lambda(f_{n},F)^{-1}\}_{n=1}^{\infty}$ is also tight.
Therefore, it has a subsequence, say $\{\lambda(f_{n},F)^{-1}\}_{n=1}^{\infty}$ itself, which converges weakly to some measure $\tau$ such that $\tau_X=\mu$.
Thus,
$$1\geq \tau(\Gamma)\geq \limsup\; \lambda(f_{n},F)^{-1}(\Gamma)=1.$$
Because of the relative saturation property,  there exists a $\cT$-measurable mapping $f$ such that $\lambda(f,F)^{-1}=\tau$ and $\lambda f^{-1}=\mu$, which implies that  $f$ is a $\cT$-measurable selection of $F$. Therefore, $\mu=\lambda f^{-1}\in D_{F}^{\cT}$ and $\overline{D_{F}^{\cF}}\subseteq D_{F}^{\cT}$.

As a result, $D_{F}^{\cT}=\overline{D_{F}^{\cF}}$, $D_{F}^{\cT}$ is closed.

\

A3. Since $F$ is compact valued, Proposition 3.8 of \cite{Sun1996} implies that $F$ is tight in the sense that for every $\epsilon > 0$, there is a compact set $K_\epsilon$ in $X$ such that the set $\{t \in T: F(t) \subseteq K_\epsilon \}$ is $\cF$-measurable and its measure is greater
than $1- \epsilon$. Hence, $D_{F}^{\cF}$ is relatively compact. By the proof of (A2), $D_{F}^{\cT}=\overline{D_{F}^{\cF}}$. Thus, $D_{F}^{\cT}$ is closed and relatively compact, and hence compact.

\

A4. Since $F$ is compact valued and $G(\cdot,y)$ is a closed subset of $F(\cdot)$ for every $y\in Y$, $G(\cdot, y)$ is also compact valued. By  (A3), $D_{F}^{\cT}$ is compact and $H(\cdot)$ is compact valued. As a result, $H$ is a compact valued correspondence from $Y$ to the compact set $D_{F}^{\cT}$. To prove that $H$ is upper hemicontinuous, it suffices to show that if $y_{n}\rightarrow y$ in $Y$ (with $y_n \ne y$ for any $n \ge 1$), $\mu_{n}\in H(y_{n})=D_{G_{y_{n}}}^{\cT}$ and $\mu_{n} \rightarrow \mu$ weakly, then $\mu\in D_{G_{y}}^{\cT}=H(y)$; see Theorem~17.11 in \cite{AB2006}.

For a fixed $n\ge 1$, since $\mu_{n}\in D_{G_{y_{n}}}^{\cT}=\overline{D_{G_{y_{n}}}^{\cF}}$, there exists a sequence $\{\mu_{n}^{m}\}_{m=1}^{\infty}$ in  $D_{G_{y_{n}}}^{\cF}$ such that $\mu_{n}^{m}$ converges weakly to $\mu_{n}$ as $m \rightarrow \infty$. For each $n\ge 1$, there exists some $m_{n}$ such that $\overline{d}(\mu_{n}^{m_{n}},\mu_{n})<\frac{1}{n}$, where $\overline{d}$ is the corresponding Prokhorov metric. Fix this $m_n$. Let $g_{n}$ be an $\cF$-measurable selection of $G_{y_{n}}$ such that $\mu^{m_n}_{n}=\lambda g_{n}^{-1}$.

Let $J(t)=\overline{\left\{(g_n(t),y_n) \right\}_{n =1}^\infty }$, where $\bN = \{1, 2, \ldots\}$ is the set of natural numbers. Then $J$ is a compact valued $\cF$-measurable correspondence from $T$ to $X\times Y$, and $D_J^{\cT}$ is compact by (A3). Since $\lambda(g_n,y_n)^{-1}$  converges weakly to $\mu\otimes \delta_y \in D_J^{\cT}$ ($\delta_y $ is the Dirac measure at $y$),
$J$ has a $\cT$-measurable selection $(g,y)$ of $J$ such that $\lambda(g,y)^{-1}=\mu\otimes\delta_y$.
Since $G_t(\cdot)$ is upper hemicontinuous for all $t\in T$, $g$ is a $\cT$-measurable selection of $G_y$ and  $\mu\in D^{\cT}_{G_y}$. Thus, $H$ is upper hemicontinuous.

\

A5. Since $G$ is $\cF$-measurable, the function $t\rightarrow G(t)(B)$ is $\cF$-measurable for every Borel subset $B \subseteq X$. Let $F$ be the correspondence from $T\rightarrow X$ such that $F(t)=supp\: G(t)$. Then $F$ is a closed valued $\cF$-measurable correspondence and $G(t)(F(t))=1$ for $\lambda$-almost all $t \in T$.

For every open set $O$ in $X$, we have
$$F^{-1}(O)=\{t:F(t)\cap O\neq\emptyset\}=\{t:G(t)(O)>0\}.$$
As a result, $F^{-1}(O)$ is $\cF$-measurable. Define
$$\mu(O)=\int_{T}G(t)(O)\rmd\lambda(t)$$
for every open set $O$ in $X$. Then
\begin{align*}
\mu(O)
& =\int_{T}G(t)(O)\rmd\lambda(t)=\int_{\{t \colon G(t)(O)>0\}} G(t)(O)\rmd\lambda(t) \\
& \leq \lambda( \{t \colon G(t)(O)>0\} )=\lambda(F^{-1}(O)).
\end{align*}
By Proposition~3.5 in \cite{KS2009}, $\mu$ belongs to the closure of $D_{F}^{\cF}$. By the proof of (A2), $\mu\in D_{F}^{\cT}$. Thus, there exists a $\cT$-measurable selection $f$ of $F$ such that $\lambda f^{-1}=\mu$.
This completes the proof.

\subsection{Proof of the necessity part of Theorem~\ref{thm-distribution}}

In the necessity part of Theorem~\ref{thm-distribution}, we work with a fixed uncountable Polish space $X$.
If $(T,\cF,\lambda)$ is purely atomic, the necessity part of Theorem~\ref{thm-distribution} holds automatically. Suppose that $T$ can be partitioned into two disjoint $\cF$-measurable parts $T_1$ and $T_2$ such that $\cF^{T_1}$ is atomless and  $\cF^{T_2}$ is purely atomic, $T=T_1\cup T_2$, and $\lambda(T_1)=1 - \gamma$ for some $\gamma \in [0, 1)$. When $(T,\cF,\lambda)$ is atomless, we have $T_2 = \emptyset$ and $\gamma =0$.

Recall that $(I,\cI,\eta)$ denotes the Lebesgue unit interval. Let $\cI_1$ be the restriction of $\cI$ on $(\gamma,1]$, and $\eta_1$ the Lebesgue measure on $(\gamma,1]$. Since $(T_1,\cF^{T_1},\lambda)$ is atomless and $\cF^{T_1}$ is countably generated, by Lemma 5 of \cite{HSS2017}, there exists a measure preserving mapping $\phi:(T_1,\cF^{T_1},\lambda)\rightarrow((\gamma,1],\cI_1,\eta_1)$ such that for any $E\in \cF^{T_1}$, there exists a set $E'\in\cI_1$ with $\lambda(E \triangle \phi^{-1}(E'))=0$. We extend the domain of the mapping $\phi$ to $T$ by letting $\phi(t)=\frac{\gamma}{2}$  for any $t\in T_2$. Then $\phi$ is an $\cF$-measurable mapping from $T$ to $[0,1]$.

\

A$1'$.\footnote{We put an additional prime for the corresponding labels in the proof of the necessity part.} For any $n\geq 1$, let $A=[-n,n]$,
$$F(t)=\{\phi(t)+i:0\leq i\leq n-1\}\cup \{-\phi(t)-i:0\leq i\leq n-1\},$$
and
$$\mu^{+}_{i}=\lambda(\phi+i)^{-1}, \qquad \mu^{-}_{i}=\lambda(-\phi-i)^{-1}.$$
Then for $0\leq i\leq n-1$, the support of $\mu^{+}_{i}$ concentrates on a subset of $(i,i+1]$ and vanishes outside,  and the support of $\mu^{-}_{i}$ concentrates on a subset of $[-i-1,-i)$ and vanishes outside. Since $\phi$ is $\cF$-measurable, $F$ is closed valued and $\cF$-measurable.

Let $\varphi$ be a Borel measurable bijection from $A$ to $X$ (see Theorem~2.12 in \cite{Parthasarathy1967}), and $G = \varphi \circ F$. Then $G$ is an $\cF$-measurable finite valued correspondence from $T$ to $X$. By the condition of (A$1'$), $D_{G}^{\cT}$ is convex. We claim that $D_{F}^{\cT}$ is also convex. Fix $0 < \alpha < 1$, and two $\cT$-measurable selections $f_1$ and $f_2$ of $F$. Let $g_i = \varphi \circ f_i$. Then $g_i$ is a $\cT$-measurable selection of $G$ for each $i$. Since $D_{G}^{\cT}$ is convex, there exists a $\cT$-measurable selection $g$ of $G$ such that $\lambda g^{-1} = \alpha \lambda g_1^{-1} + (1 - \alpha) \lambda g_2^{-1}$. Since $\varphi$ is a Borel isomorphism from $A$ to $X$, there exists a $\cT$-measurable selection $f$ of $F$ such that $g = \varphi \circ f$, which implies that $\lambda f^{-1} = \alpha \lambda  f_1^{-1} + (1 - \alpha) \lambda f_2^{-1}$. As a result, $D_{F}^{\cT}$ is convex.

As $\mu_i^{+}, \mu_i^{-} \in D_{F}^{\cT}$ for each $i$,
$$\mu=\frac{1}{2n}\left(\sum_{0\leq i\leq n-1}\mu_i^{+}+ \sum_{0\leq i\leq n-1}\mu_i^{-}\right)  \in D_{F}^{\cT}.$$
There exists a $\cT$-measurable selection $f$ of $F$ and $2n$ $\cT$-measurable disjoint sets $E^{+}_{i},E^{-}_{i}$ for $0\leq i\leq n-1$ such that
\begin{enumerate}
\item $\mu=\lambda f^{-1}$;
\item $\lambda\left(\cup_{1\leq i\leq n-1} (E_{i}^{+}\cup E_{i}^{-})\right)=1$; and
\item $$f(t)=\begin{cases}
\phi(t)+i, & t\in E_{i}^{+},\\
-\phi(t)-i, & t\in E_{i}^{-}.
\end{cases}
$$
\end{enumerate}

Let $E^{1+}_i=E^{+}_i\cap T_1$ and  $E^{1-}_i=E^{-}_i \cap T_1$ for $0\leq i \leq n-1$. For any set $E\in\cF^{T_1}$, there exists some set $E'\in\cI_1$ such that $\lambda(E\triangle\phi^{-1}(E'))=0$. Notice that on $\left(\cup_{0\leq i\leq n-1}(\gamma+i,1+i]\right)\cup \left(\cup_{0\leq i\leq n-1}[-1-i, -\gamma-i)\right)$,  $\mu$ is uniformly distributed with the density $\frac{1}{2n}$ with respect to the Lebesgue measure on the union of the $2n$ disjoint intervals. Since
\begin{align*}
\lambda(E^{1+}_0\cap E)
& = \lambda(E^{1+}_0\cap\phi^{-1}(E'))=\lambda(f^{-1}(E'))=\mu(E') \\
& =\frac{1}{2n}\mu_0^{+}(E')=\frac{1}{2n}\lambda(\phi^{-1}(E'))=\frac{1}{2n}\lambda(E),
\end{align*}
we have
$$\lambda(E_{0}^{1+})=\lambda(E_{0}^{1+}\cap T_1)=\frac{1}{2n}\lambda(T_1)=\frac{1-\gamma}{2n}.$$
Therefore, $E_{0}^{1+}$ is independent of $\cF^{T_1}$ under the probability  measure $\lambda^{T_1}$ and of measure $\frac{1-\gamma}{2n}$ under $\lambda$. Similarly, we could prove that $E_{i}^{1+}$ and $E_{i}^{1-}$ have the same property for all $0\leq i\leq n-1$. That is, $\cF^{T_1}$ admits an asymptotic independent supplement in $\cT^{T_1}$ under $\lambda^{T_1}$. By Lemma~\ref{lem-neq}, $\cT^{T_1}$ is nowhere equivalent to $\cF^{T_1}$ under $\lambda^{T_1}$. Since $\cF$ is atomic on $T_2$, $\cT$ is nowhere equivalent to $\cF$ under $\lambda$.

\

A$2'$. Fix $n \ge 1$. Consider the correspondence $F$ constructed in the proof of (A$1'$). Let $\Pi = \{D_j\}_{0 \le j \le 2n - 1}$ be a $\cT$-measurable partition of $T_2$ such that $\lambda(D_j) = \frac{\gamma}{2n} $ for $0 \le j \le 2n-1$. Define a sequence of functions $\{f_k\}$ as follows: for each $k \ge 1$,
$$f_k(t) =
\begin{cases}
\phi(t)+i, & t \in D_i, 0 \le i \le n-1; \\
-\phi(t)-i, & t \in D_{i+n}, 0 \le i \le n-1; \\
\phi(t)+i, & \phi(t) \in (\gamma + \frac{j}{2nk}(1 - \gamma), \gamma + \frac{j + 1}{2nk}(1 - \gamma)] \\
 & \mbox{ for some } j = 2nk' + i, 0 \le k' < k, 0 \le i \le n-1; \\
-\phi(t)-i, & \phi(t) \in (\gamma + \frac{j}{2nk}(1 - \gamma), \gamma + \frac{j + 1}{2nk}(1 - \gamma)] \\
& \mbox{ for some } j = 2nk' + n+i, 0 \le k' < k, 0 \le i \le n-1.
\end{cases}
$$
Then $f_k$ is a $\cT$-measurable selection of $F$. It can be checked that the sequence $\{\lambda  f_k^{-1}\}_{k=1}^\infty$ converges weakly to the Borel probability measure $\mu$ as defined in the proof of (A$1'$).

Since $X$ is an uncountable Polish space, as in \cite[Section~6]{RSY1995}, $X$ contains a compact subset $K$ that is homeomorphic to the Cantor set $C$; see \cite[p.~11]{Parthasarathy1967}. In addition, there is a continuous onto mapping from $C$ to $A$ (see \cite[p.~127]{HY1961}). Let $\beta_2$ be a continuous onto mapping from $K$ to $A$. By the Borel cross section theorem (see \cite[Theorem~4.2]{Parthasarathy1967}), there is a Borel measurable mapping $\beta_1$ from $A$ to $K$ such that $\beta_2(\beta_1(a))$ = $a$ for all $a\in A$.

Let $G(t)=\beta_1\circ F(t)$, and $g_k = \beta_1 \circ f_k$ for each $k \ge 1$. Then $G$ is an $\cF$-measurable correspondence and $g_k$ is a $\cT$-measurable selection of $G$. Since $F(t)$ only contains finitely many elements for every $t \in T$, $G(t)$ is finite, and thus a closed set. By the condition of (A$2'$), $D_G^{\cT}$ is weakly closed. By the compactness of $K$, the sequence $\{\lambda g_k^{-1}\}_{k=1}^\infty$ has a weakly convergence subsequence (say itself) with a limt $\tau \in \cM(K)$.
Since $D_G^{\cT}$ is closed, there exists some $\cT$-measurable selection $g$ of $G$ such that $\lambda g^{-1} = \tau$. Let $f = \beta_2 \circ g$. Then $\lambda f^{-1} = \lambda g^{-1} \beta_2^{-1} = \tau \beta_2^{-1}$. As $f(t) = \beta_2 (g(t)) \in \beta_2 ( G(t)) = \beta_2 (\beta_1 (F(t))) = F(t)$ for each $t \in T$, $f$ is a $\cT$-measurable selection of $F$. Since $f_k = \beta_2 \circ g_k$, $\lambda f_k^{-1} = \lambda g_k^{-1} \beta_2^{-1}$. Since  $\beta_2$ is continuous and the sequence $\{\lambda g_k^{-1}\}_{k=1}^\infty$ converges weakly to $\tau$, we know that  $\{\lambda f_k^{-1}\}_{k=1}^\infty$ converges weakly to $\tau \beta_2^{-1}$. The fact that both $\mu$ and $\tau \beta_2^{-1}$ are the weak limit of the sequence $\{\lambda f_k^{-1}\}_{k=1}^\infty$ implies  that $\mu = \tau \beta_2^{-1} = \lambda   f^{-1}$.

Following the proof of (A$1'$),  $\cT^{T_1}$ is nowhere equivalent to $\cF^{T_1}$ under $\lambda^{T_1}$. Since the restriction $\cF$ to $T_2$ is purely atomic, $\cT$ is nowhere equivalent to $\cF$ under $\lambda$.

\

A$3'$. Notice that the correspondences $F$ and $G$ in the proof of (A$2'$) are both finite valued, and hence are compact valued. In addition, if $D_{F}^{\cT}$/$D_{G}^{\cT}$ is compact, then it is closed. As a result, we indeed have proved this claim in the proof of (A$2'$).

\

A$4'$. Fix $n\geq 1$. Let $A=[-n,n]$.
Define a sequence of correspondences $\{G_k\}_{k\in\bN}$ as follows.
\begin{enumerate}
\item If $\phi(t)\in[0,\gamma)$, then
$$G_k(t)= \{\phi(t)+i:0\leq i\leq n-1\}\cup \{-\phi(t)-i:0\leq i\leq n-1\};$$

\item if
$$\phi(t)\in \left( \gamma+(1-\gamma)\left(\frac{k'}{2k}+\frac{i}{2nk} \right), \gamma+(1-\gamma) \left(\frac{k'}{2k}+\frac{i+1}{2nk} \right)\right]$$
for some $0\leq i \leq n-1$, $0\leq k' \leq 2k-1$ with $k'$ being even, then
\begin{align*}
G_k(t) =
& \bigg\{ 2n\left[ \phi(t)-\gamma-(1-\gamma)\left( \frac{k'}{2k}+\frac{i}{2nk} \right) \right]+\gamma+(1-\gamma)\frac{k'}{2k}+i, \\
& -2n \left[ \phi(t)-\gamma-(1-\gamma)\left(\frac{k'}{2k}+\frac{i}{2nk} \right) \right]-\gamma-(1-\gamma)\frac{k'}{2k}-i \bigg\};
\end{align*}

\item if
$$\phi(t)\in \left( \gamma+(1-\gamma)\left(\frac{k'}{2k}+\frac{i}{2nk} \right), \gamma+(1-\gamma)\left(\frac{k'}{2k}+\frac{i+1}{2nk} \right) \right]$$
for some $0\leq i \leq n-1$, $0\leq k' \le 2k-1$ with $k'$ being odd, then
\begin{align*}
G_k(t) =
& \bigg\{ 2n \left[ \phi(t)-\gamma-(1-\gamma)\left(\frac{k'}{2k}+\frac{i}{2nk} \right) \right]+\gamma+(1-\gamma)\frac{k'-1}{2k}+i, \\
& -2n\left[ \phi(t)-\gamma-(1-\gamma)\left(\frac{k'}{2k}+\frac{i}{2nk} \right) \right]-\gamma-(1-\gamma)\frac{k'-1}{2k}-i \bigg\}.
\end{align*}
\end{enumerate}
Let
$$G_0(t) = \left\{ \phi(t)+i \colon 0\leq i\leq n-1 \right\} \cup \left\{-\phi(t)-i \colon 0\leq i\leq n-1 \right\} .$$
Then $G_k$ is $\cF$-measurable for each $k \ge 0$.

Let $\Pi = \{D_j\}_{0 \le j \le 2n - 1}$ be a $\cT$-measurable partition of $T_2$ such that $\lambda(D_j) = \frac{\gamma}{2n} $ for $0 \le j \le 2n-1$.  For $k \ge 1$, let $g_k(t)$ be defined as follows.
\begin{enumerate}
\item If $t \in D_i$ for $0 \le i \le n-1$, then $g_k(t) = \phi(t)+i$;

\item If $t \in D_{i+n}$ for $0 \le i \le n-1$, then $g_k(t) = -\phi(t)-i$;

\item if
$$\phi(t)\in \left( \gamma+(1-\gamma)\left(\frac{k'}{2k}+\frac{i}{2nk} \right), \gamma+(1-\gamma)\left(\frac{k'}{2k}+\frac{i+1}{2nk} \right) \right]$$
for some $0\leq i \leq n-1$, $0\leq k' \leq 2k-1$ with $k'$ being even, then
$$g_k(t) = 2n \left[ \phi(t)-\gamma-(1-\gamma)\left(\frac{k'}{2k}+\frac{i}{2nk} \right) \right]+\gamma+(1-\gamma)\frac{k'}{2k}+i; $$

\item if
$$\phi(t)\in \left( \gamma+(1-\gamma)\left(\frac{k'}{2k}+\frac{i}{2nk} \right), \gamma+(1-\gamma)\left(\frac{k'}{2k}+\frac{i+1}{2nk} \right) \right]$$
for some $0\leq i \leq n-1$, $0\leq k' \leq 2k-1$ with $k'$ being odd, then
$$g_k(t)=-2n \left[ \phi(t)-\gamma-(1-\gamma)\left(\frac{k'}{2k}+\frac{i}{2nk} \right) \right]-\gamma-(1-\gamma)\frac{k'-1}{2k}-i.$$
\end{enumerate}
Then $g_k$ is a $\cT$-measurable selection of $G_k$, and $\mu = \lambda g_k^{-1}$ for any $k$, where $\mu$ is defined in the proof of (A$1'$).

Let $Y=\{0,1,\frac{1}{2},\cdots\}$ be endowed with the usual metric. As in the proof of (A$2'$), we fix a compact subset $K \subseteq X$ that is homeomorphic to the Cantor set $C$, a Borel measurable mapping $\beta_1$ from $A$ to $K$, and a continuous onto mapping $\beta_2$ from $K$ to $A$ such that $\beta_2(\beta_1(a))$ = $a$ for all $a\in A$. Let $\tilde{G}(t, 0) = \beta_2^{-1} \circ G_0(t)$, $\tilde{G}(t,\frac{1}{k}) = \beta_2^{-1} \circ G_k(t)$ and $\tilde{g}_k = \beta_1 \circ g_k$  for each $k \ge 1$. Then the correspondence $\tilde{G} \colon T \times Y \to X$ satisfies the condition of (A$4'$). In addition, $\beta_2(\tilde{g}_k(t)) = \beta_2(\beta_1 (g_k(t))) = g_k(t) \in G_k(t)$, which implies that $\tilde{g}_k$ is a selection of $\tilde{G}(\cdot,\frac{1}{k})$. Denote $\tau = \mu \beta_1^{-1}$. As $\lambda g_k^{-1} = \mu$, $\lambda \tilde{g}_k^{-1} = \mu  \beta_1^{-1} = \tau$. We claim that $\tau \in D_{\tilde{G}_0}^\cT$. If $\tau \notin D_{\tilde{G}_0}^\cT$, then there exists an open neighbourhood $U$ of $D_{\tilde{G}_0}^\cT$ such that $\tau \notin U$. For each $k$, $\tau \in D_{\tilde{G}_{\frac{1}{k}}}^\cT$, which implies that $D_{\tilde{G}_{\frac{1}{k}}}^\cT$ is not included in the open set $U$. This contradicts with the condition of (A$4'$). As a result, $\tau \in D_{\tilde{G}_0}^\cT$. Then there exists some $\cT$-measurable selection $\tilde{g}_0$ of $\tilde{G}(\cdot, 0)$ such that $\lambda \tilde{g}_0^{-1} = \tau$. Let $g_0 = \beta_2 \circ \tilde{g}_0$. As $g_0(t) = \beta_2 (\tilde{g}_0(t)) \in \beta_2 ( \tilde{G}(t, 0)) = \beta_2 (\beta_2^{-1} (G_0(t))) = G_0(t)$, $g_0$ is a $\cT$-measurable selection of $G_0$. Then $\lambda g_0^{-1} = \lambda \tilde{g}_0^{-1} \beta_2^{-1} = \tau \beta_2^{-1}  = \mu \beta_1^{-1} \beta_2^{-1} = \mu$. That is, there exists a $\cT$-measurable selection $g_0$ of $G_0$ which induces the measure $\mu$.
Note that $G_0$ is the same as the correspondence $F$ constructed in the proof of (A$1'$). Following the proof of (A$1'$),  $\cT^{T_1}$ is nowhere equivalent to $\cF^{T_1}$ under $\lambda^{T_1}$. Since the restriction $\cF$ to $T_2$ is purely atomic, $\cT$ is nowhere equivalent to $\cF$ under $\lambda$.

\

A$5'$. Fix $n\geq 1$. Let $A=[-n,n]$, and $\varphi$ a Borel measurable bijection from $A$ to $X$. Define
$$G_1(t)=\frac{1}{2n}\left(\sum_{0\leq i\leq n-1}\delta_{\varphi(\phi(t)+i)}+  \sum_{0\leq i\leq n-1}\delta_{\varphi(-\phi(t)-i)} \right)$$
for $t \in T$, where $\delta_x$ is the Dirac measure at $x \in X$. Then $G_1$ is $\cF$-measurable. By the condition of (A$5'$), there exists a $\cT$-measurable mapping $f_1$ from $t$ to $X$ such that
\begin{enumerate}
\item $f_1(t) \in \varphi (F(t)) = \left\{ \varphi(\phi(t)+i):0\leq i\leq n-1 \right\} \cup \left\{ \varphi(-\phi(t)-i):0\leq i\leq n-1 \right\}$, where $F$ is the correspondence constructed in the proof of (A$1'$);
\item for every Borel subset $B_1$ of $X$, $\lambda f_1^{-1}(B_1)=\int_{T}G_1(t)(B_1)\rmd\lambda(t)$.
\end{enumerate}

Define
$$G(t) = \frac{1}{2n}\left(\sum_{0\leq i\leq n-1}\delta_{\phi(t)+i}+  \sum_{0\leq i\leq n-1}\delta_{-\phi(t)-i} \right)$$
for $t \in T$, where $\delta_a$ is the Dirac measure at $a \in A$.
Let $f= \varphi^{-1} \circ f_1$. Then $f$ is a $\cT$-measurable selection of $F$.

For every Borel set $B\subseteq A$,
\begin{align*}
\lambda f^{-1}(B)
& =\lambda f_1^{-1} (\varphi (B)) =\int_{T}G_1(t)(\varphi(B))\rmd\lambda(t) =\int_{T}G(t)(B)\rmd\lambda(t) \\
& =\frac{1}{2n}\sum_{0\leq i\leq n-1}\left(\lambda(\phi+i)^{-1} + \lambda(-\phi-i)^{-1} \right) (B)=\mu(B),
\end{align*}
where $\mu$ is the probability measure constructed in the proof of (A$1'$). Following the proof therein, $\cT$ is nowhere equivalent to $\cF$ under $\lambda$.

\subsection{Proof of Theorem~\ref{thm-rcd}}
As noted earlier, the necessity part of Theorem~\ref{thm-rcd} follows from the necessity part of Theorem~\ref{thm-distribution} by taking $\cG$ as the trivial $\sigma$-algebra. Thus, we shall only prove the sufficiency part of Theorem~\ref{thm-rcd}. As discussed in Subsection~\ref{subsec-proof S1}, we can assume without loss of generality that $(T,\cF,\lambda)$ is an atomless probability space.

Since $\cF$ is countably generated, it is easy to see that the weak topology on $\cR^{\cF}$ is semimetrizable, and hence is first countable (see \cite[Lemma~3.3]{AB2006}). By \cite[Theorem~2.40]{AB2006} (resp. \cite{BH1998}), the following sequential convergence suffices for our aim when considering the closedness (resp. compactness) on $\cR^{\cF}$:  a sequence $\{\phi_n\}$ in $\cR^\cF$ is said to weakly converge to some $\phi\in\cR^\cF$ ($\phi_n\Longrightarrow \phi$) if for every  bounded Carath\'eodory function $c: T\times X\to \bR$,
      $$\lim_{n\to\infty}\int_T \left[\int_X c(t,x) \phi_n(t,\rmd x) \right]\rmd \lambda(t)= \int_T \left[\int_X c(t,x) \phi(t,\rmd x) \right] \rmd\lambda(t).$$

By Theorem~2.1.3 in \cite{Castaing2004}, we have the following lemma.
\begin{lem}\label{lem-rcd weak}
Suppose that $\phi_n,\phi \in \cR^{\cF}$ for $n \ge 1$. Then the sequence $\phi_n\Longrightarrow \phi$ if and only if
      $$
      \lim_{n\to\infty}\int_E \left[\int_X c(x) \phi_n(t,\rmd x) \right]\rmd \lambda(t)= \int_E \left[\int_X c(x) \phi(t,\rmd x) \right] \rmd\lambda(t)
      $$
for every $E\in \cF$ and $c\in C_b(X)$.
\end{lem}

\begin{proof}[Proof of the sufficiency part of Theorem~\ref{thm-rcd}]

Given a sub-$\sigma$-algebra $\cG$ of $\cF$ and an $\cF$-measurable correspondence $F$ from $T$ to $X$, let $\psi$ be a Borel measurable mapping from $T$ to the interval $Y =[0,1]$ such that $\cG$ is the $\sigma$-algebra generated by $\psi$ (see \citet[Theorem 6.5.5]{Bogachev2007}). Then the correspondence $F_1$ as defined by $F_1(t) =\{\psi(t)\} \times F(t)$ for any $t \in T$ is $\cF$-measurable from $T$ to $Y\times X$. Let $\kappa = \lambda \psi^{-1}$, $g=(\psi,f)$ be a $\cT$-measurable selection of $F_1$ and $\nu^{(\psi,f)}=\lambda g^{-1}$. Then $\nu^{(\psi,f)}_Y=\kappa$. Since $X$ and $Y$ are both Polish spaces, there exists a  family of Borel probability measures $\{\vartheta^{(\psi,f)}(y,\cdot)\}_{y\in Y}$ ($\kappa$-a.e. uniquely determined) in $\cM(X)$, which is the disintegration of $\nu^{(\psi,f)}$ with respect to $\kappa$ on $Y$.

Define a transition probability $\mu^f$ from $T$ to $\cM(X)$ as $\mu^f(t,B)=\vartheta^{(\psi,f)}(\psi(t),B)$ for each $t\in T$ and Borel set $B\subseteq X$. Let
$$\cR_\psi=\{\mu^f:f \mbox{ is a } \cT\mbox{-measurable selection of } F\}.
$$
We shall prove that $\cR_\psi$ coincides with $\cR_F^{(\cT,\cG)}$.

Fix an arbitrary $\cT$-measurable selection $f$ of $F$. For any $E\in \cG$, there exists a Borel subset $E'\subseteq Y$ such that $\lambda(E\triangle \psi^{-1}(E'))=0$. For any Borel subset $B$ of $X$,
\begin{align*}
& \quad \int_E  \mu^f(t,B) \rmd\lambda(t)= \int_E  \vartheta^{(\psi,f)}(\psi(t),B) \rmd\lambda(t)  \\
&  = \int_{E'} \vartheta^{(\psi,f)}(y,B) \rmd\kappa(y) =  \nu^{(\psi,f)}(E'\times B)  =  \lambda(E\cap f^{-1}(B)).
\end{align*}
Thus, $\mu^f(\cdot, B)= \mathbb{E} \left[1_B(f)|\cG \right] =\mu^{f|\cG}(\cdot, B)$. By the essential uniqueness of regular conditional distribution, we have  $\cR_\psi=\cR_F^{(\cT,\cG)}$.

\

B1. Suppose that $f_1$ and $f_2$ are two $\cT$-measurable selections of $F$ and $0\leq \alpha\leq 1$. By Theorem~\ref{thm-distribution}, $D_{F_1}^\cT$ is convex, which implies that there exists a $\cT$-measurable selection $f$ of $F$ such that
$$\lambda (\psi,f)^{-1}=\alpha\lambda(\psi,f_1)^{-1}+ (1-\alpha)\lambda(\psi,f_2)^{-1}.$$
That is, $\nu^{(\psi,f)}=\alpha \nu^{(\psi,f_1)} + (1-\alpha)\nu^{(\psi,f_2)}$.

To show $\mu^{f|\cG}=\alpha \mu^{f_1|\cG} + (1-\alpha) \mu^{f_2|\cG}$, it is equivalent to show  $\mu^{f}=\alpha \mu^{f_1} + (1-\alpha) \mu^{f_2}$.
For any $E\in \cG$, and Borel subsets $E'\subseteq Y$ and $B\subseteq X$ such that $\lambda(E\triangle \psi^{-1}(E'))=0$, we have
\begin{align*}
& \quad \alpha\int_E \mu^{f_1}(t,B) \rmd\lambda(t) + (1-\alpha)\int_E  \mu^{f_2}(t,B) \rmd\lambda(t)\\
& = \alpha \nu^{(\psi,f_1)}(E'\times B)  + (1-\alpha) \nu^{(\psi,f_2)}(E'\times B) \\
& = \nu^{(\psi,f)}(E'\times B) \\
& = \int_E  \mu^{f}(t,B) \rmd\lambda(t).
\end{align*}
Therefore, $\cR_F^{(\cT,\cG)}$ is convex.

\

B2. Let $\{f_n\}_{n=1}^\infty$ be a sequence of $\cT$-measurable selections of $F$. Assume that $\{\mu^{f_n|\cG}\}_{n=1}^\infty$ is weakly convergent to some $\mu$ in $\cR^\cG$. Let $\nu$ be a Borel probability measure on $Y\times X$ such that
$$\nu(E'\times B)=\int_{\psi^{-1}(E')} \mu(t,B)\rmd\lambda(t),$$
where $E'$ and $B$ are Borel sets in $Y$ and $X$.

For any $c \in C_b(Y\times X)$, we have
\begin{align*}
\lim_{n\to \infty} \int_{Y\times X} c(y,x) \rmd\nu^{(\psi,f_n)}
& =\lim_{n\to \infty}\int_Y \int_X c(y,x) \vartheta^{(\psi,f_n)}(y,\rmd x)\rmd \kappa(y)\\
& = \lim_{n\to \infty}\int_T \int_X c(\psi(t),x) \mu^{f_n}(t,\rmd x)\rmd \lambda(t)\\
& = \int_T \int_X c(\psi(t),x) \mu(t,\rmd x)\rmd \lambda(t)\\
& = \int_{Y\times X} c(y,x) \rmd \nu(y,x),
\end{align*}
which implies that $\nu^{(\psi,f_n)}$ converges weakly to $\nu$.
The first equality is the disintegration; the second equality is changing of variables; the third equality is due to the weak convergence; and the last equality is due to the definition of $\nu$.

Since $\nu^{(\psi,f_n)}\in D_{F_1}^\cT$ for $n \ge 1$ and $D_{F_1}^\cT$ is closed, $\nu\in D_{F_1}^\cT$. That is, there exists a $\cT$-measurable selection $f$ of $F$ such that $\nu =\lambda (\psi,f)^{-1}=\nu^{(\psi,f)}$. Therefore, $\mu= \mu^{f|\cG}\in\cR_F^{(\cT,\cG)}$, $\cR_F^{(\cT,\cG)}$ is closed.

\

B3. Let $\{f_n\}_{n=1}^\infty$ be a sequence of $\cT$-measurable selections of $F$. We need to show that there exists a subsequence of $\{\mu^{f_n|\cG}\}_{n=1}^{\infty}$ that converges weakly to some element in $\cR_{F}^{(\cT,\cG)}$.

Since $D_{F_1}^\cT$ is compact, there is a subsequence of $\{f_n\}_{n=1}^\infty$, say $\{f_n\}_{n=1}^\infty$ itself, and a $\cT$-measurable selection $f$ of $F$ such that $\lambda (\psi,f_n)^{-1}$ converges weakly to $\lambda (\psi,f)^{-1}$ as $n$ goes to infinity. For any $E\in \cG$, any Borel subsets $E'\subseteq Y$ and $B\subseteq X$ such that $\lambda(E\triangle \psi^{-1}(E'))=0$, and any bounded continuous function $c$ on $X$ with an absolute bound $M$,
\begin{align*}
\int_{E} \left[\int_X c(x) \mu^{f_n}(t,\rmd x) \right]\rmd \lambda(t)
& = \int_{E'\times X} c(x) \rmd \nu^{(\psi,f_n)}(y,x) \\
& = \int_{Y\times X} 1_{E'}(y) c(x) \rmd \nu^{(\psi,f_n)}(y,x).
\end{align*}
For any $\epsilon > 0$, Lusin's Theorem (\citet[Theorem 7.1.13]{Bogachev2007}) implies that there exists a continuous function $h_\epsilon$ from $Y=[0, 1]$ to $[0, 1]$ such that $\kappa (\{ y \in Y: h_\epsilon(y) \ne 1_{E'}(y)\}) < \epsilon$. Since $\lambda (\psi,f_n)^{-1}$ converges weakly to $\lambda (\psi,f)^{-1}$,
$$\int_{Y\times X} h_\epsilon(y) c(x) \rmd \nu^{(\psi,f_n)}(y,x) \to \int_{Y\times X} h_\epsilon(y) c(x) \rmd \nu^{(\psi,f)}(y,x).
$$
Since
$$\left| \int_{Y\times X} 1_{E'}(y) c(x) \rmd \nu^{(\psi,f_n)}(y,x) - \int_{Y\times X} h_\epsilon(y) c(x) \rmd \nu^{(\psi,f_n)}(y,x) \right| < M\epsilon
$$
and
$$\left| \int_{Y\times X} 1_{E'}(y) c(x) \rmd \nu^{(\psi,f)}(y,x) - \int_{Y\times X} h_\epsilon(y) c(x) \rmd \nu^{(\psi,f)}(y,x) \right| < M\epsilon,
$$
we have
$$ \int_{Y\times X} 1_{E'}(y) c(x) \rmd \nu^{(\psi,f_n)}(y,x) \to  \int_{Y\times X} 1_{E'}(y) c(x) \rmd \nu^{(\psi,f)}(y,x).
$$
That is,
$$\int_{E} \left[\int_X c(x) \mu^{f_n}(t,\rmd x) \right]\rmd \lambda(t) \to   \int_E \int_X c(x) \mu^f(t,\rmd x)\rmd \lambda(t).
$$
By Lemma \ref{lem-rcd weak}, $\mu^{f_n|\cG}$ converges weakly to $\mu^{f|\cG}$ as $n$ goes to infinity, and hence $\cR_F^{(\cT,\cG)}$ is compact.

\

B4. Suppose that $z_n \to z_0$ in $Z$ and $\mu^{g_n|\cG}$ converges weakly to $\mu$ as $n$ goes to infinity, where $g_n$ is a $\cT$-measurable selection of $G_{z_n}$ for each $n\ge 1$. We will show that $\mu\in \cR_{G_{z_0}}^{(\cT,\cG)}$.

Let $\nu$ and $\nu_n$ be the Borel probability measures on $Y\times X$ such that
$$\nu(E'\times B)=\int_{\psi^{-1}(E')}\mu(t,B)\rmd \lambda (t),$$
and
$$\nu_n(E'\times B)=\int_{\psi^{-1}(E')}\mu^{g_n|\cG}(t,B)\rmd \lambda (t),$$
where $E'$ and $B$ are Borel sets in $Y$ and $X$ respectively. Following  the proof of (B2), one can show that there is a subsequence of $\{\nu_n\}$, say $\{\nu_n\}$ itself, which converges weakly to $\nu$.

Define $\Psi(t,z)= \{\psi(t)\} \times G(t,z)$. Then $(\psi,g_n)$ is a $\cT$-measurable selection of $\{\psi\} \times G_{z_n}$ and $\nu_n=\lambda (\psi,g_n)^{-1}$.
Since $D^\cT_{\Psi_z}$ is upper hemicontinuous from $Z$ to $\cM(Y\times X)$,  $\nu\in D^\cT_{\Psi_{z_0}}$. That is, there is a $\cT$-measurable selection $g$ of $G_{z_0}$ such that $\nu=\lambda (\psi,g)^{-1}$.  It is easy to see that $\mu=\mu^{g|\cG}\in \cR_{G_{z_0}}^{(\cT,\cG)}$, and hence the correspondence $H$ is upper hemicontinuous.

\

B5. Since $G$ is $\cG$-measurable, the function $t\rightarrow G(t)(B)$ is $\cG$-measurable for each Borel set $B$ in $X$.
Let ${\hat G}$ be the correspondence from $T\rightarrow X$ such that ${\hat G}(t)=supp\: G(t)$, and ${\hat G}_1 (t)=\{\psi(t)\} \times {\hat G} (t)$ for ant $t \in T$. Then ${\hat G}_1$ is a closed valued $\cG$-measurable correspondence.

Let $\nu$ be the Borel probability measure on $X\times Y$ such that for any Borel sets $B_X$ in $X$ and $B_Y$ in $Y$,
$$\nu(B_X\times B_Y)=\int_{\psi^{-1}(B_Y)}G(t)(B_X)\rmd\lambda(t).$$
Then, for any open sets $O_X$ in $X$ and $O_Y$ in $Y$, we have
\begin{align*}
\nu(O_X\times O_Y)
& = \int_{\psi^{-1}(O_Y)}G(t)(O_X)\rmd\lambda(t) = \int_{\psi^{-1}(O_Y)\cap\{t:G(t)(O_X)>0\}}G(t)(O_X)\rmd\lambda(t) \\
& \le \lambda \left(\psi^{-1}(O_Y)\cap\{t:G(t)(O_X)>0\} \right) =\lambda \left({\hat G}_1^{-1}(O_X\times O_Y) \right).
\end{align*}
By Proposition~3.5 in \cite{KS2009}, $\nu$ belongs to the closure of $D_{{\hat G}_1}^{\cG}$; that is, $\nu\in D_{{\hat G}_1}^{\cT}$. There exists a $\cT$-measurable selection $g$ of ${\hat G}$ such that $\lambda(\psi,g)^{-1}=\nu$. For any $E\in \cG$, Borel subsets $E'\subseteq Y$ and $B\subseteq X$ such that $\lambda(E\triangle \psi^{-1}(E'))=0$,
\begin{align*}
&\int_{E}G(t)(B)\rmd\lambda(t)
 = \int_{\psi^{-1}(E')}G(t)(B)\rmd\lambda(t) =\nu(E'\times B)\\
& = \lambda\left(\psi^{-1}(E')\cap g^{-1}(B)\right) = \lambda(g^{-1}(B)\cap E) = \int_E \mathbb{E} \left[ 1_B (g) | \cG \right] \rmd\lambda(t).
\end{align*}
Hence $\mu^{g|\cG}=G$. This completes the proof.
\end{proof}

Note that in the proof of the sufficiency part of Theorem \ref{thm-rcd} (B5), the mapping $g$ constructed indeed satisfies the following property: for $\lambda$-almost all $t\in T$, $g(t)\in supp\, G(t)$. Below, we show that this property is also implied by the conclusion of (B5).

\begin{coro}
Given $G\in\cR^{\cG}$ and a $\cT$-measurable mapping $g$ from $t$ to $X$. If $\mu^{g|\cG}=G$, then for $\lambda$-almost all $t\in T$, $g(t)\in supp\, G(t)$.
\end{coro}

\begin{proof}
Let $id_T$ be the identity mapping on $T$, and $\tau$ the probability measure on $T \times X$ such that for any $E \in \cG \otimes \cB(X)$, $\tau(E) = \int_T G(t, E_t) \rmd\lambda(t)$.
Since $\mu^{g|\cG}=G$, it follows from the definition of $\mu^{g|\cG}$ that for any $D\in \cG$ and any Borel set $B$ in $\cB(X)$,
\begin{equation}\label{equa-1}
\int_{T} 1_{D \times B} (t, g(t)) \rmd\lambda(t) = \int_T 1_D(t) G(t,B) \rmd\lambda(t).
\end{equation}
Hence, the probability measures $\tau$ and $\lambda \left(id_T, g\right)^{-1}$ on $\cG \otimes \cB(X)$ agree on the measurable rectangles. Since the class of measurable rectangles is closed under finite intersections (i.e., a $\pi$-system), Dynkin’s Lemma (\cite[p.~136]{AB2006}) implies that $\tau$ and $\lambda \left(id_T, g\right)^{-1}$ are the same.

Define a mapping $c$ from $T\times X$ to $\bR$ as $c(t,x)=1_{supp\, G(t)}(x)$. Then $c$ is positive and $\cG \otimes \cB(X)$-measurable. We have
$$\int_{T} c(t,g(t)) \rmd\lambda(t) = \int_{T \times X} c(t,x) d \tau =
\int_T\int_X c(t,x) G(t,\rmd x) \rmd\lambda(t) = 1.$$
Thus, $c(t,g(t))=1$ for $\lambda$-almost all $t\in T$.
\end{proof}

\section{Proof of Theorem~\ref{thm-lg}}\label{sec-game proof}

Notice that the probability space $(T,\cF,\lambda)$ may not be atomless. There exist two disjoint $\cF$-measurable subsets $T_1$ and $T_2$ such that $T_1\cup T_2= T$, and $\lambda|_{T_1}$ is the atomless part of $\lambda$, while $\lambda|_{T_2}$ is the purely atomic part of $\lambda$. Let $\lambda(T_1) = 1 - \gamma$. We shall assume $0 \le \gamma < 1$ to avoid triviality.

We first provide an example. Fix $n\geq 2$. For $1\le i\le n$, denote
$$A^{i+}=\{(0,\cdots,0,a_i,0,\cdots,0)\colon 0 \le a_i \leq 2\},$$
which is the set of $n$-dimensional vectors such that only the $i$-th entry is nonzero. Similarly, let
$$A^{i-}=\{(0,\cdots,0,a_i,0,\cdots,0) \colon -2\leq a_i \le 0\}.$$
Denote $A=\cup_{1\le i\le n}(A^{i+} \cup A^{i-})$. Then $A$ is a compact absolute retract\footnote{A compact metric space $X$ is called an absolute retract if for any metric space $Y$ containing $X$, there is a 
continuous mapping $r \colon Y \to X$ such that the restriction of $r$ to $X$ is the identity map on $X$ (such a mapping is called a retraction).}
of $\bR^n$.\footnote{Let $C$ be the cube $[-2, 2]^n \subseteq \bR^n$. Then $A$ is a subset of $C$. The following function $r$ is a retraction from $C$ to $A$. For any $(x_1, \ldots, x_n) \in C$, let
$$r(x_1, \ldots, x_n) =
\begin{cases}
(0, \ldots, 0, |x_i| - |x_j|, 0, \ldots, 0), &  \mbox {if } x_i \ge 0 \mbox { and } |x_i| \ge |x_j| \ge |x_k| \mbox{ for any } k \neq i, j,  \\
(0, \ldots, 0, -|x_i| + |x_j|, 0, \ldots, 0), &  \mbox {if } x_i < 0 \mbox { and } |x_i| \ge |x_j| \ge |x_k| \mbox{ for any } k \neq i, j;
\end{cases}
$$
where only the $i$-th component could be nonzero. By Theorem 2-34 in \cite[p.63]{HY1961}, the unit cube in $\bR^n$ is an absolute retract, so is $C$. Then by Lemma~2.1 in \cite{Hanner1951}, $A$ is also an absolute retract.} The Borel $\sigma$-algebras on $A$ is denoted by $\cB_0$.

For any Borel set $E  \subseteq \bR$ and $c\in \bR$, denote $E+c=\{a+c:a\in E\}$.  Let
$$E^{i+}=\{(0,\cdots,0,a,0,\cdots,0):a\in E \cap \bR_+\}$$
and
$$E^{i-}=\{(0,\cdots,0,-a,0,\cdots,0):a\in E\cap \bR_+\}$$
such that  $a$ and $-a$ are in the $i$-th entry. For any $a\in \bR_+$, let
$$a^{i+}=(0,\cdots,0,a,0,\cdots,0),$$
$$a^{i-}=(0,\cdots,0,-a,0,\cdots,0),$$
where $a$ and $-a$ are in the $i$-th entry.

For any set $E\in\mathcal{B}_{0}$, there exists a unique choice of $2n$ sets $\{D_i,E_i\}_{1\le i \le n} \subseteq  \cB([0,2])$ such that $E = \cup_{1 \le i\le n} (D_i^{i+} \cup E_i^{i-})$. We define a Borel probability measure  $\nu_1$  on $A$ as
$$\nu_{1}(E)=\frac{1}{2n(1 - \gamma)} \left\{ \sum_{1\le i \le n}\eta(D_i \cap [1,2 - \gamma)) +  \sum_{1\le i \le n}\eta(E_i \cap [1,2 - \gamma)) \right\}.$$
Let $\nu_0$ be a convex combination of $\nu_1$ and the Dirac measure concentrated at $(0,\cdots,0)$:
$$\nu_0 = (1 - \gamma) \nu_1 + \gamma \delta_{(0,\cdots,0)}.$$

Let $f^{i+}$ be a function from $A^{i+}\times[0,1]$ to $\bR$ as follows:
$$
f^{i+}((0,\cdots,0,a_i,0,\cdots,0),b) =
$$
$$
\begin{cases}
0, & \hspace{-55pt} \mbox{if } b=0 \mbox{ or } a_i \in [0,1] \\
& \hspace{-55pt} \mbox{ or } a_i-1=kb \mbox{ for some } k\in \bN; \\
\frac{1}{2} \min\{a_i-1-2nkb-(2i-2)b, 2nkb+(2i-1)b+1-a_i\}, & \\
& \hspace{-250pt} \mbox{if } a_i-1\in \big( 2nkb+(2i-2)b,  2nkb+(2i-1)b\big) \mbox{ for some } k\in \bN; \\
0, & \hspace{-55pt} \mbox{otherwise}.
\end{cases}
$$

Let $f^{i-}$ be a function from $A^{i-}\times[0,1]$ to $\bR$ as follows:
$$
f^{i-}((0,\cdots,0,-a_i,0,\cdots,0),b) =
$$
$$
\begin{cases}
0, &  \hspace{-55pt} \mbox{if } b=0 \mbox{ or } a_i \in [0, 1] \\
& \hspace{-55pt} \mbox{ or } a_i-1=kb \mbox{ for some } k\in \bN; \\
\frac{1}{2} \min\{a_i-1-2nkb-(2i-1)b, 2nkb+2ib+1-a_i\}, & \\
& \hspace{-220pt} \mbox{if } a_i-1\in \big( 2nkb+(2i-1)b, 2nkb+2ib\big) \mbox{ for some } k\in \bN; \\
0, & \hspace{-55pt} \mbox{otherwise}.
\end{cases}
$$

The following figures illustrate the functions $f^{i+}$ and $f^{i-}$ with $b=\frac{1}{8}$ and $n=2$.
\begin{figure}[htb!]
\centering
\includegraphics[width=\textwidth]{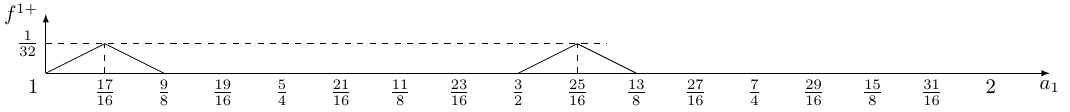}
\caption{The graph of $f^{1+}((a_1, 0),b)$}
\end{figure}
\begin{figure}[htb!]
\centering
\includegraphics[width=\textwidth]{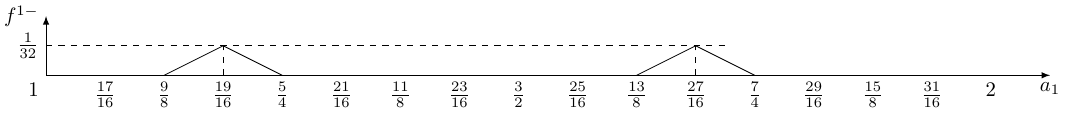}
\caption{The graph of $f^{1-}((-a_1, 0),b)$}
\end{figure}
\begin{figure}[htb!]
\centering
\includegraphics[width=\textwidth]{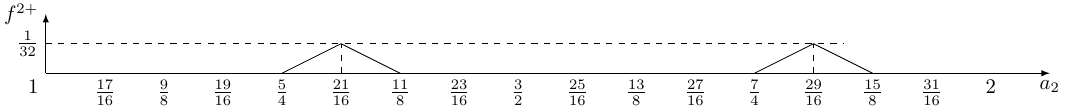}
\caption{The graph of $f^{2+}((0, a_2),b)$}
\end{figure}
\begin{figure}[htb!]
\centering
\includegraphics[width=\textwidth]{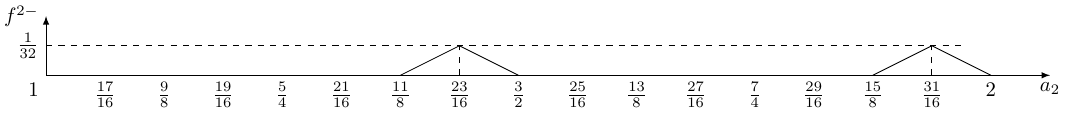}
\caption{The graph of $f^{2-}((0, -a_2),b)$}
\end{figure}
\FloatBarrier

Define a function $f$ from $A\times [0,1]$ to $\bR$ as follows. For $i=1,2,\cdots n$,
$$ f(a,b) =
\begin{cases}
f^{i+}(a,b), & \text{if }a\in A^{i+}; \\
f^{i-}(a,b), & \text{if }a\in A^{i-}.
\end{cases}
$$
It is easy to see that the function $f$ is continuous on $A\times [0,1]$. Let $f_0$ be a function from $\cM(A)$ to $[0,1]$ defined as $f_0(\nu)=\frac{1}{2n}d(\nu_0,\nu)$ for any $\nu \in \cM(A)$, where $d(\cdot, \cdot)$ is the Prohorov metric on $\cM(A)$.

By Lemma 6 of \cite{HSS2017}, there exists a measure-preserving mapping $\phi \colon (T_1,\cF^{T_1},\lambda) \to ([1, 2 - \gamma),\cB_1,\eta_1)$ such that for any $E\in\cF^{T_1}$, there exists a set $E'\in\cB_1$ with $\lambda(E\triangle\phi^{-1}(E'))=0$, where $\cB_1$ is the Borel $\sigma$-algebra on $[1, 2 - \gamma)$, and $\eta_1$ is the Lebesgue measure on $\cB_1$.

Now we are ready to describe the example.

\begin{exam}\label{exam}
Fix an integer $n\ge2$. Let $(T,\cT,\lambda)$ be the agent space and $A$ the action space. Define an $\cF$-measurable large game $G\colon T\to\cU_A$ as follows:
$$G(t)(a,\nu)=\begin{cases}
f(a,f_0(\nu))-\vert \phi(t)-\|a\|\vert, & \mbox{ if } t\in T_1; \\
-\|a\|,  & \mbox{ if }  t \in T_2
\end{cases}$$
for any $a\in A$ and $\nu\in\cM(A)$, where $\|\cdot\|$ is the usual Euclidean norm in $\bR^n$.
\end{exam}

\begin{lem}\label{lem-example1}
Suppose that there exists a $\cT$-measurable Nash equilibrium $g$ for the large game $G$ above. Then there exists a $\cT$-measurable partition $\{E_j,D_j\}_{1\le j \le n}$ of $T_1$ such that for each $j=1,\ldots,n$,
\begin{enumerate}
\item $\lambda^{T_1}(D_j)=\lambda^{T_1}(E_j)=\frac{1}{2n}$,
\item $D_j$ and $E_j$ are independent of $\cF^{T_1}$ under the probability measure $\lambda^{T_1}$.
\end{enumerate}
\end{lem}

\begin{proof}
Notice that $G_t$ is a continuous function on $A\times\cM(A)$ for any $t\in T$. Moreover, it is clear that $G$ is $\cF$-measurable. Suppose that $g$ is a $\cT$-measurable Nash equilibrium of $G$. Let $\vartheta= \lambda g^{-1}$. We shall prove that $\vartheta = \nu_0$.

Suppose that $\vartheta\neq \nu_0$. Let $b_0=\frac{1}{2n}\dist(\vartheta,\nu_0)$. Then $0<b_0\le\frac{1}{2n}$. It is obvious that $g(t) = (0,\cdots,0)$ for each agent $t \in T_2$. Below, we shall fix an agent $t \in T_1$ with $\phi(t)\neq kb_0$ for any $k\in\bN$.\footnote{Without loss of generality, we can ignore the set $\{t \colon \phi(t) =  kb_0 \mbox{ for some } k\}$, which is of probability zero.}

If $g(t)\in A^{i+}$, then for any $(0,\cdots,0,x,0,\cdots,0)\in A^{i+}$ with $x\neq\phi(t)$,
$$G(t)(x^{i+},\vartheta)-G(t)(\phi(t)^{i+},\vartheta) = f^{i+}(x^{i+},f_0(\vartheta))-|\phi(t)-x| -f^{i+}(\phi(t)^{i+},f_0(\vartheta))  <0.
$$
The inequality is due to the observation that  $f^{i+}$ is a Lipschitz functions in terms of the $i$-th coordination with the Lipschitz constant $\frac{1}{2}$ for all $i=1,\cdots, n$. Thus, $g(t) = \phi(t)^{i+}$. Similarly, we could show that $g(t) = \phi(t)^{i-}$ if $g(t)\in A^{i-}$.

Next, we determine the column that the nonzero coordination of $g(t)$ locates at; that is, the value of $i$ and the sign of $g(t)$. For agent $t$, there is a unique pair $(k,i')$ with $k \in \bN$ and $i' \in \{0,\cdots,2n-1\}$ such that $\phi(t)-1 \in ((2nk+i')b_0, (2nk+i'+1)b_0)$.
\begin{enumerate}
\item If $i'$ is even, then the value of $G(t)(g(t),\vartheta)$ is positive only if $i = \frac{i'}{2}+1$ and $g(t)\in A^{i+}$.
\item If $i'$ is odd, then the value of $G(t)(g(t),\vartheta)$ is positive only if $i = \frac{i'+1}{2}$ and $g(t)\in A^{i-}$.
\end{enumerate}
Therefore,
$$g(t) =
$$
$$\begin{cases}
\phi(t)^{(\frac{i'}{2}+1)+}, & \text{if } \phi(t) \in ((2nk+i')b_0 + 1, (2nk+i'+1)b_0+1) \mbox{ for some even } i'; \\
\phi(t)^{(\frac{i'+1}{2})-}, & \text{if } \phi(t) \in ((2nk+i')b_0 + 1, (2nk+i'+1)b_0+1) \mbox{ for some odd } i'; \\
(0,\cdots,0), &\text{if } t \in T_2.
\end{cases}
$$

Below, we show that $\dist(\vartheta,\nu_0)$ is at most $(2n-1)b_0$. Let $\epsilon=(2n-1)b_0$.

For $i=1,\cdots,n$, let $W^{i+}$ be the support of $\vartheta$ on $A^{i+}$. By the analysis above, the set $W^{i+}$ is the union of finite disjoint intervals, which are denoted by $W^{i+}_1,\ldots,W^{i+}_m$ in the increasing order. The distance between $W^{i+}_\ell$ and $W^{i+}_{\ell+1}$ is $(2n-1)b_0$ for $\ell=1,\ldots,m-1$. It is clear that the length of $W^{i+}_\ell$ is $b_0$ for $\ell=1,2,\ldots,m-1$, and the length of $W^{i+}_m$ is at most $b_0$. For any Borel set $E\in \cB$, consider the set $E^{i+}$. Without loss of generality, we could assume that $E^{i+}$ does not contain any endpoint of the subintervals $\{W^{i+}_\ell\}_{1 \le \ell \le m}$. For $1\le\ell\leq m$, let $E^{i+}_\ell=W^{i+}_\ell\cap E^{i+}$. Then $E^{i+}_\ell,E^{i+}_\ell+b_0,\ldots,E^{i+}_\ell+(2n-1)b_0$ are all disjoint, and $E^{i+}_\ell+tb_0$ is included in $(E^{i+})^\epsilon$ for $t=0,\ldots,2n-1$, where $(E^{i+})^\epsilon$ is the $\epsilon$-neighborhood of $E^{i+}$. We have
\begin{align*}
    		& \vartheta(E^{i+})   \\
={} 		& \sum_{\ell=1}^{m-1} \vartheta(E^{i+}_\ell)+ \vartheta(E^{i+}_m) =  2n\sum_{\ell=1}^{m-1} \nu_0(E^{i+}_\ell)+ \vartheta(E^{i+}_m) \le  2n\sum_{\ell=1}^{m-1} \nu_0(E^{i+}_\ell)+ b_0 \\
={} 		& \sum_{\ell=1}^{m-1} \Big(\nu_0(E^{i+}_\ell)+ \nu_0(E^{i+}_\ell+b_0) + \cdots + \nu_0(E^{i+}_\ell+(2n-1)b_0)\Big) + b_0  \\
\le{} 	& \nu_0\big((E^{i+})^\epsilon\big)+b_0.
\end{align*}
Similarly, we could prove that $\vartheta(E^{i-}) \le \nu_0\big((E^{i-})^\epsilon\big)+b_0$ for $i=1,\cdots,n-1$.

Let $W^{n-}$ be the support of $\vartheta$ on $A^{n-}$. The set $W^{n-}$ is the union of finite disjoint intervals, which are denoted by $W^{n-}_1,\ldots,W^{n-}_m$ in the increasing order. The distance between $W^{n-}_\ell$ and $W^{n-}_{\ell+1}$ is $(2n-1)b_0$ for $\ell=1,2,\ldots,m-1$. In addition, the distance between $1^{n-}$ and $W^{n-}_{1}$ is also $(2n-1)b_0$. The length of $W^{n-}_\ell$ is $b_0$ for $\ell=1,2,\ldots,m-1$, and the length of $W^{n-}_m$ is at most $b_0$. Take a Borel set $E \subseteq [1, 2 - \gamma)$. Without loss of generality, we may assume that $E$ does not contain any endpoint of the subintervals $\{W^{n-}_\ell\}_{1 \le \ell \le m}$. For $1\le\ell\leq m$, let $E^{n-}_\ell=W^{n-}_\ell\cap E^{n-}$. Then $E^{n-}_\ell,E^{n-}_\ell-b_0,\ldots,E^{n-}_\ell-(2n-1)b_0$ are all disjoint, and $(E^{n-}_\ell-tb_0)$ is included in $(E^{n-})^\epsilon$ for $t=0,\ldots,2n-1$. We have
\begin{align*}
    		& \vartheta(E^{n-})   \\
={} 		& \sum_{\ell=1}^{m} \vartheta(E^{n-}_\ell) = 2n\sum_{\ell=1}^{m} \nu_0(E^{n-}_\ell)  \\
={} 		& \sum_{\ell=1}^{m} \Big(\nu_0(E^{n-}_\ell)+ \nu_0((E^{n-}_\ell-b_0)) + \cdots + \nu_0(E^{n-}_\ell-(2n-1)b_0)\Big)  \\
\le{} 	& \nu_0\big((E^{n-})^\epsilon \big).
\end{align*}

Given any Borel set $C\subseteq A$ such that $(0,\cdots,0)\notin C$. Then $C=\cup_{1\leq k\leq n} (C^{k+}\cup C^{k-})$, where $C^{k+}\subseteq A^{k+}$ and $C^{k-}\subseteq A^{k-}$. We have
\begin{align*}
\vartheta(C)
& = \sum_{k=1}^{n} [\vartheta(C^{k+}) + \vartheta(C^{k-})]\\
& \leq \sum_{k=1}^{n}\{\nu_0\big((C^{k+})^\epsilon \big)+b_0\} + \sum_{k=1}^{n-1}\{\nu_0\big((C^{k-})^\epsilon \big)+b_0\} + \nu_0\big((C^{n-})^\epsilon \big)\\
& \leq \nu_0(C^\epsilon) + (2n-1)b_0 = \nu_0(C^\epsilon) + \epsilon.
\end{align*}
Similarly, we could prove the case with $(0,\cdots,0)\in C$. Hence, $d(\vartheta,\nu_0)\leq (2n-1)b_0 =\frac{2n-1}{2n} d(\vartheta,\nu_0)$, which is a contradiction. Therefore, we prove $\vartheta=\nu_0$ as claimed in the beginning of the proof of this lemma.

From now on, we shall work with the case that $\vartheta=\nu_0$. Then $f_0(\vartheta) = 0$, and hence $f(a,f_0(\vartheta))=0$ for any $a\in A$ and $G(t)(a,\vartheta)=-|\phi(t)-\|a\||$ for any $t\in T_1$. The best response correspondence is
$$H(t) =
\begin{cases}
\{(0,\cdots,0,\phi(t),0,\cdots,0)_i,(0,\cdots,0,-\phi(t),0,\cdots,0)_i\}_{1\le i \le n}, & t \in T_1, \\
\{(0,\cdots, 0)\}, & t \in T_2,
\end{cases}
$$
where $(0,\cdots,0,\phi(t),0,\cdots,0)_i$ means that $\phi(t)$ is in the $i$-th entry. By the definition of Nash equilibria, $g(t)\in H(t)$ for $\lambda$-almost all $ t \in T$.

For any $C\in\cF^{T_1}$, by the choice of $\phi$, there exists a set $C_1\in\cB_1$ such that $\lambda(C\triangle\phi^{-1}(C_1))=0$. Define
$$D_i=\{t \in T_1\colon g(t)=(0,\cdots,0,\phi(t),0,\cdots, 0)_i\}$$
and
$$E_i=\{t \in T_1\colon g(t)=(0,\cdots,0,-\phi(t), 0,\cdots, 0)_i\}_i.$$
Thus, we have
\begin{align*}
\lambda (D_i\cap C)    & =\lambda(D_i\cap \phi^{-1}(C_1))=\lambda(g\in C_1^{i+}) =\nu_0(C_1^{i+}) \\
                & =\frac{1}{2n}\eta(C_1)=\frac{1}{2n} \lambda (\phi\in C_1)=\frac{1}{2n} \lambda (C),
\end{align*}
and hence $\lambda^{T_1}(D_i)=\frac{1}{2n}$. Therefore, $D_i$ is independent of $\cF^{T_1}$ under $\lambda^{T_{1}}$ for $i=1,\cdots,n$. Similarly, we could show the same result for $E_i$, $i=1,\cdots,n$. In particular, $\{D_1, E_1, \ldots, D_{n}, E_n\}$ is a $\cT$-measurable partition of $T_1$. This completes the proof.
\end{proof}

Let $K$ be an uncountable compact metric space, $A$ an uncountable compact absolute retract of $\bR^n$, and $\mathcal{U}_{K}$ and  $\mathcal{U}_{A}$ the  spaces of real-valued continuous functions on $K\times\mathcal{M}(K)$ and $A\times\mathcal{M}(A)$, respectively. Lemma \ref{lem-change action} below extends the result in Section 6 of \cite{RSY1995} from the case $A = [-1,1]$ to the case of a general compact absolute retract here.
It shows that if pure strategy equilibria exist in all large games with a fixed uncountable compact metric action space, then for any $n\ge 1$, the existence result can be extended to large games with any uncountable compact absolute retract of $\bR^n$ as the action space.\footnote{The action spaces for the large games in Example 9 of \cite{HSS2017} consist of some parallel line segments. Since those action sets are not compact absolute retracts, we are not able to use the large games over there to prove the necessity of our Theorem \ref{thm-lg} here.} Though the argument is very much similar to the one in \cite[Section 6]{RSY1995}, we provide a detailed proof of Lemma \ref{lem-change action} for the sake of completeness.

\begin{lem}\label{lem-change action}
If for any $\cF$-measurable game $F \colon T \rightarrow \mathcal{U}_{K}$, there exists a $\cT$-measurable function $f\colon T\rightarrow K$ such that $f$ is a Nash equilibrium of the game $F$, then for any $\cF$-measurable game $G \colon T\rightarrow\mathcal{U}_{A}$, there exists a $\cT$-measurable function $g \colon T\rightarrow A$ such that $g$ is a Nash equilibrium of the game $G$.
\end{lem}

\begin{proof}
Since $K$ is uncountable and compact, as discussed in Section~6 of \cite{RSY1995}, $K$ contains a compact subset $M$ which is homeomorphic to the Cantor set $C$; see \cite[p.~11]{Parthasarathy1967}. In addition, for the compact absolute retract $A \subseteq \bR^l$, there is a continuous onto mapping from $C$ to $A$ (see \cite[p.~127]{HY1961}). Let $\varphi$ be a continuous onto mapping from $M$ to $A$. By Tietze's extension theorem (see \cite{Hanner1951}),
$\varphi$ can be extended to a continuous onto mapping $\Upsilon$ from $K$ to $A$. By the Borel cross section theorem (see \cite[Theorem~4.2]{Parthasarathy1967}), there is a Borel measurable mapping $\Xi$ from $A$ to $K$ such that $\Upsilon(\Xi(a))$ = $a$ for all $a\in A$.

For a given $\cF$-measurable game $G \colon T\rightarrow\mathcal{U}_{A}$, define an $\cF$-measurable game $F \colon T \rightarrow \mathcal{U}_{K}$ as follows: for any $t \in T$, $x \in K$ and $\nu \in \cM(K)$,
$$F_t(x, \nu) = G_t\big( \Upsilon(x), \nu \Upsilon^{-1} \big).
$$
By the assumption, there exists a $\cT$-measurable mapping $f \colon T \rightarrow K$ such that $f$ is a Nash equilibrium of the game $F$. Denote $\nu' = \lambda f^{-1}$.

Define a $\cT$-measurable mapping $g$ from $T$ to $A$ as $g(t) = \Upsilon(f(t))$. We need to show that $g$ is a Nash equilibrium of the game $G$. Note that
$$\lambda g^{-1} = \lambda f^{-1} \Upsilon^{-1} = \nu' \Upsilon^{-1}.$$
For player~$t \in T$,
$$G_t(g(t), \lambda g^{-1}) = G_t(\Upsilon(f(t)), \nu' \Upsilon^{-1}) = F_t(f(t), \nu'),
$$
and
$$G_t(a, \lambda g^{-1}) = G_t(\Upsilon(\Xi(a)), \nu' \Upsilon^{-1}) = F_t(\Xi(a), \nu').
$$
Since $f$ is a Nash equilibrium of the game $F$, for $\lambda$-almost all $t \in T$ and all $a \in A$,
$$F_t(f(t), \nu') \ge F_t(\Xi(a), \nu'),$$
which implies that
$$G_t(g(t), \lambda g^{-1}) \ge G_t(a, \lambda g^{-1}).$$
This completes the proof.
\end{proof}

Now we are ready to prove Theorem~\ref{thm-lg}.

\begin{proof}[Proof of Theorem~\ref{thm-lg}]
\

As mentioned in Section~\ref{sec-lg}, we only need to prove the necessity part. By the condition of Theorem~\ref{thm-lg}, for any $\cF$-measurable game $F \colon T\rightarrow\mathcal{U}_{K}$, there exists a $\cT$-measurable function $f \colon T\rightarrow K$ such that $f$ is a Nash equilibrium of the game $F$. Consider the action space $A$ constructed in Example~\ref{exam} above. In particular, $A$ is a compact absolute retract. By Lemma~\ref{lem-change action}, for any $\cF$-measurable game $G \colon T\rightarrow\mathcal{U}_{A}$, there exists a $\cT$-measurable function $g \colon T\rightarrow A$ such that $g$ is a Nash equilibrium of the game $G$. By Lemma~\ref{lem-example1}, $\cF^{T_1}$ admits an asymptotic independent supplement within $\cT^{T_1}$ under $\lambda^{T_1}$. By Lemma~\ref{lem-neq}, $\cT^{T_1}$ is nowhere equivalent to $\cF^{T_1}$ under $\lambda^{T_1}$. Since $\cF$ is purely atomic on $T_2$ under $\lambda$, $\cT$ is nowhere equivalent to $\cF$ under $\lambda$.
\end{proof}

{\small
\singlespacing

\end{document}